\documentclass[10pt, twoside, reqno]{amsart}
\usepackage{xcolor, amstext, amsfonts, amssymb, amsbsy, latexsym}
\usepackage[normalem]{ulem}
\usepackage{enumerate}
\usepackage{hyperref}
\usepackage[T1]{fontenc}
\usepackage{xy, hhline}
\newtheorem{lemma}{Lemma}[section]
\newtheorem{theo}[lemma]{Theorem}
\newtheorem{prop}[lemma]{Proposition}

\newtheorem{cor}[lemma]{Corollary}

\theoremstyle{definition}
\newtheorem{defin}[lemma]{Definition}

\newtheorem{remark}[lemma]{Remark}

\numberwithin{equation}{section}

\newenvironment{eq}{\begin{equation}}{\end{equation}}

\newcommand{\Char}{\mathop{\rm char}}

\newcommand{\FF}{\mathbb{F}}

\newcommand{\RR}{\mathbb{R}}

\newcommand{\mult}{\circ}
\newcommand{\Sym}{{\mathcal S}}

\newcommand{\al}{\alpha}
\newcommand{\be}{\beta}
\newcommand{\ga}{\gamma}
\newcommand{\la}{\lambda}
\newcommand{\de}{\delta}

\newcommand{\ov}[1]{\overline{#1}}

\newcommand{\tr}{\mathop{\rm tr}}
\newcommand{\Aut}{\mathop{\rm Aut}}

\newcommand{\G}{{\rm G}_2}
\newcommand{\SL}{{\rm SL}}

\newcommand{\matr}[4]{\left(\begin{array}{cc}
#1 & #2 \\
#3 & #4 \\
\end{array}\right)}

\newcommand{\OO}{\mathbf{O}}

\newcommand{\uu}{\mathbf{u}}
\newcommand{\vv}{\mathbf{v}}
\newcommand{\cc}{\mathbf{c}}
\newcommand{\zero}{\mathbf{0}}
\newcommand{\n}{\mathbf{n}}

\newcommand{\corr}[1]{{#1}}

\begin{document}
\title[On linear equations over split-octonions]{On linear equations over split-octonions}


\author{Artem Lopatin}
\address{Artem Lopatin\\
Universidade Estadual de Campinas (UNICAMP), 651 Sergio Buarque de Holanda, 13083-859 Campinas, SP, Brazil}
\email{dr.artem.lopatin@gmail.com (Artem Lopatin)}

\author{Alexandr N. Zubkov}
\address{Alexandr N. Zubkov\\
Sobolev Institute of Mathematics, Omsk Branch, Pevtzova 13, 644043 Omsk, Russia
}
\email{a.zubkov@yahoo.com (Alexandr N. Zubkov)}

\thanks{The work was supported by FAPESP 2023/17918-2}

\begin{abstract}
Over an algebraically closed field, we describe the affine varieties of solutions to the linear equations \corr{$(ax)b=c$} and $a(bx)=c$ over the split-octonions. We also determine the dimensions of the solution sets of arbitrary linear monomial equations in the split-octonions. Moreover, we show that if a linear monomial equation over the split-octonions with nonzero constant term has at least two solutions, then it necessarily possesses an invertible solution.


\end{abstract}

\maketitle

\section{Introduction}

Unless stated otherwise, let $\FF$ be a field of arbitrary characteristic $p = \Char\FF \geq 0$. All vector spaces and algebras are assumed to be defined over~$\FF$.

\subsection{Equations over octonions}

The problem of solving polynomial equations has historically been regarded as one of the central problems in mathematics, and its study contributed to the development of Algebraic Geometry and several other mathematical disciplines. Polynomial equations have been investigated not only over fields, but also over matrix algebras, as well as over the algebras of quaternions, octonions, and related structures. 

Rodríguez-Ordó\~nez~\cite{Rodriguez-Ordonez_2007} proved that every polynomial equation of positive degree over the algebra $\mathbf{A}_{\RR}$ of \emph{Cayley numbers} (that is, the division algebra of real octonions) containing only a single term of highest degree has at least one solution. Furthermore, Wang, Zhang, and Zhang~\cite{Wang_Zhang_2014} obtained an explicit algorithm for solving quadratic equations of the form $x^2 + bx + c = 0$ over $\mathbf{A}_{\RR}$, together with criteria determining whether such an equation admits one, two, or infinitely many solutions.

In general, an \emph{octonion algebra} (or \emph{Cayley algebra}) $\mathbf{C}$ over the field $\FF$ is an $8$-dimensional unital non-associative alternative algebra equipped with a nonsingular quadratic multiplicative form $\n : \mathbf{C} \to \FF$, called the \emph{norm} (see Section~\ref{section_general_octonion} for details). The norm $\n$ is said to be \emph{isotropic} if $\n(a)=0$ for some nonzero $a \in \mathbf{C}$, and \emph{anisotropic} otherwise. When $\n$ is anisotropic, the algebra $\mathbf{C}$ is a division algebra. If $\n$ is isotropic, then there exists a unique octonion algebra $\OO_{\FF}$ over $\FF$ with isotropic norm (see Theorem~1.8.1 of~\cite{Springer_Veldkamp_book_2000}); this algebra is referred to as the \emph{split-octonion algebra}. Moreover, if $\FF$ is algebraically closed, then every octonion algebra over $\FF$ is isomorphic to the split-octonion algebra $\OO_{\FF}$ (see Lemma~\ref{lemma_alg_closed} below).

Since an octonion algebra $\mathbf{C}$ is alternative, it is therefore power-associative, i.e., the subalgebra generated by a single element is associative. Hence, given $a \in \mathbf{C}$, the expression $a^n$ may be written without specifying the placement of brackets in the product. Flaut and Shpakivskyi~\cite{Flaut_Shpakivskyi_2015} studied the equation $x^n = a$ over real octonion division algebras. For an octonion division algebra $\mathbf{C}$ over an arbitrary field $\FF$, Chapman~\cite{Chapman_2020_JAA} presented a complete method for finding the solutions of the polynomial equation $a_n x^n + a_{n-1} x^{n-1} + \cdots + a_1 x + a_0 = 0$ over $\mathbf{C}$. Moreover, Chapman and Vishkautsan~\cite{Chapman_Vishkautsan_2022}, working over a division algebra $\mathbf{C}$, determined the solutions of the polynomial equation $(a_n c)x^n + (a_{n-1} c)x^{n-1} + \cdots + (a_1 c)x + (a_0 c) = 0$ and discussed the solutions of the polynomial equation $(c a_n)x^n + (c a_{n-1})x^{n-1} + \cdots + (c a_1)x + (c a_0) = 0$. Chapman and Levin~\cite{Chapman_Levin_2023} described a method for finding so-called ``alternating roots'' of polynomials over an arbitrary division Cayley--Dickson algebra. 
\corr{Chapman and Vishkautsan~\cite{Chapman_Vishkautsan_2025} investigated conditions under which, for a root $a$ of a polynomial $f(x)$ over a general Cayley–Dickson algebra, there exists a polynomial $g(x)$ such that $f(x)=g(x)(x-a)$.}
Working over the split-octonions over an algebraically closed field, Lopatin and Rybalov~\cite{Lopatin_Rybalov_2025} solved all polynomial equations in which all coefficients except the constant term are scalar. As a consequence, the $n$-th roots of a split-octonion were computed.

The split-octonions have numerous applications in physics. For example, the Dirac equation, i.e., the equation of motion of a free spin $1/2$ particle such as an electron or proton, can be expressed in terms of split-octonions (see~\cite{Koplinger_2006, Gogberashvili_2006, Gogberashvili_2006_Dirac, Koplinger_2007}). Further applications of split-octonions occur in electromagnetic theory (see~\cite{Chanyal_2017_CommTP, Chanyal_Bisht_Negi_2013, Chanyal_Bisht_Negi_2011}), geometrodynamics (see~\cite{Chanyal_2015_RepMP}), unified quantum theories (see~\cite{Krasnov_2022, Bisht_Dangwal_Negi_2008, Castro_2007}), and special relativity (see~\cite{Gogberashvili_Sakhelashvili_2015}).

The case of polynomial equations over an arbitrary algebra was recently considered by Illmer and Netzer in~\cite{Illmer_Netzer_2024}, where conditions were established that guarantee the existence of a common solution to $n$ polynomial equations in $n$ variables, together with an application to polynomial equations over $\mathbf{A}_{\RR}$. 

Linear equations over matrices have been extensively studied in \cite{Eq_BK_79, Eq_BK_80, Eq_B_99, Eq_T_00, Eq_L_06, Eq_FLLT_19}, among many other works. 
The principal questions addressed in these works were 
\begin{enumerate}
\item[$\bullet$] determining conditions under which a solution to a given linear equation exists;
\item[$\bullet$] describing the general form of the solution.
\end{enumerate}
\noindent{}For particular matrix equations, further information can be found in the following recent surveys:
\begin{enumerate}
\item[$\bullet$] for the equation $AXB=C$, see~\cite{wang1};

\item[$\bullet$] for the generalized Sylvester equation $AX - YB = C$, see~\cite{wang2};

\item[$\bullet$] for the system of equations $AX=C$ and $XB=D$, see~\cite{wang3};

\item[$\bullet$]  for the system of equations $A_1 X B_1 = C_1$ and $A_2 X B_2 = C_2$, see~\cite{wang4}.
\end{enumerate}

\subsection{Results}

In this paper we consider the linear equations $ax = c$, $(ax)b = c$, and $a(bx) = c$ over the split-octonion algebra $\OO$, assuming that $\FF$ is algebraically closed. Note that over a division octonion algebra these equations can be solved easily, and for nonzero $a,b$ there exists a unique solution. Over an algebraically closed field, however, the situation is drastically different. Namely, let $\Omega_d$ denote the set of all affine subspaces (i.e., flats) in $\OO \cong \FF^8$ of dimension $1 \le d \le 8$, i.e.,
\[
\Omega_d = \{\, \mathcal{V} + a \mid \mathcal{V} \subset \OO \text{ is an }\FF\text{-subspace with } \dim \mathcal{V} = d \text{ and } a \in \OO \}.
\]
Moreover, for convenience, set \corr{$\Omega_0 =  \{\{a\}\,|\,a\in\OO\}$} and $\Omega_{-1} = \{\emptyset\}$. Given $Y \in \Omega_r$ for some $r \in \{-1,0,1,\ldots,8\}$, we define the dimension of $Y$ to be $\dim Y = r$.

The set $X$ of all solutions of one of the above linear equations with nonzero $a,b$ is either empty, consists of a single element, or satisfies $X \in \Omega_r$, where
\begin{enumerate}
\item[$\bullet$]  $r = 4$ when we consider the equation $ax = c$ (see Corollary~\ref{cor_AXisB});
\item[$\bullet$]  $r \in \{4,5,7\}$ when we consider the equation $(ax)b = c$ (see Corollary~\ref{cor_AXBisC});
\item[$\bullet$]  $r \in \{4,6,8\}$ when we consider the equation $a(bx) = c$ (see Corollary~\ref{cor_ABXisC}).
\end{enumerate}

\noindent{}Moreover, the set $X$ of all solutions of the general linear monomial equation $w(x)=c$, where $w(x)$ is a product of a variable $x$ and coefficients $a_1,\ldots,a_m \in \OO$ (see Section~\ref{section_monom} for the definition), is either
\begin{enumerate}
\item[$\bullet$] empty, or
\item[$\bullet$] consists of a single element, or
\item[$\bullet$] lies in $\Omega_r$ for some $4 \le r \le 7$, or
\item[$\bullet$] coincides with $\OO$,
\end{enumerate}
see Corollary~\ref{cor_minom_eq}. As a consequence, in Corollary~\ref{cor_invertible} we show that if a linear monomial equation over octonions with nonzero constant term has at least two solutions, then it possesses an invertible solution. We also study how the dimension of the solution set changes under degeneration of a linear monomial equation in Corollary~\ref{cor_degenerate}.

To solve a linear equation $ax = c$, we consider the action of an automorphism $g \in \Aut(\OO)$ on this equation, yielding $(ga)(gx) = gc$, such that $ga$ is a canonical non-invertible octonion (see Definition~\ref{def_canon}). Similarly, to solve the equation $(ax)b = c$ (or $a(bx) = c$, respectively), we consider the action of an automorphism $g \in \Aut(\OO)$ on the equation so that $(ga, gb)$ is a pair of canonical non-invertible octonions (see Definition~\ref{def_canon}). Hence, our solution method is based on the classification of pairs of octonions obtained by Lopatin and Zubkov in~\cite{LZ_1} (see Theorem~\ref{theo_O2_canon} below), where the split-octonion algebra $\OO$ is referred to as the octonion algebra. Note that the problem of separating $\Aut(\OO)$-orbits on pairs of octonions remains open, but closed $\Aut(\OO)$-orbits in $\OO^n$ can be distinguished by polynomial invariants (see~\cite{schwarz1988}, \cite{zubkov2018}, \cite{LZ_2} for details). We provide explicit solutions for the equations $(ax)b = c$ and $a(bx) = c$ with non-invertible canonical $(a,b)$ in Theorems~\ref{theo_AXBisC} and \ref{theo_ABXisC}.


In Sections~\ref{section_O}--\ref{section_degenerate}, with the exception of Section~\ref{section_general_octonion}, we assume that the field $\FF$ is algebraically closed. In Section~\ref{section_O} we explicitly define the octonion algebra $\OO$. Additional properties of octonions, together with key notations for certain sets of octonions, are provided in Section~\ref{section_def}. The solutions of the equations $ax = c$, $(ax)b = c$, and $a(bx) = c$ are studied in Sections~\ref{section_AXisC}, \ref{section_AXBisC}, and \ref{section_ABXisC}, respectively. Finally, in Section~\ref{section_monom} we investigate the solution sets of linear monomial equations, and in Section~\ref{section_degenerate} we examine degenerations of linear equations.

\section{Octonions}\label{section_O}

The definitions in this section are adopted from~\cite{LZ_2}.

\subsection{Split-octonions} 

The {\it split-octonion algebra} $\OO=\OO(\FF)$, also known as the {\it split Cayley algebra}, is the vector space of all matrices

$$a=\matr{\al}{\uu}{\vv}{\be}\text{ with }\al,\be\in\FF \text{ and } \uu,\vv\in\FF^3,$$%
together with the multiplication given by the following formula:
$$a a'  =
\matr{\al\al'+ \uu\cdot \vv'}{\al \uu' + \be'\uu - \vv\times \vv'}{\al'\vv +\be\vv' + \uu\times \uu'}{\be\be' + \vv\cdot\uu'},\text{ where } a'=\matr{\al'}{\uu'}{\vv'}{\be'},$$%
$\uu\cdot \vv = u_1v_1 + u_2v_2 + u_3v_3$ and $\uu\times \vv = (u_2v_3-u_3v_2, u_3v_1-u_1v_3, u_1v_2 - u_2v_1)$. For short, we denote  $\cc_1=(1,0,0)$, $\cc_2=(0,1,0)$,  $\cc_3=(0,0,1)$, $\zero=(0,0,0)$ from $\FF^3$. Consider the following basis of $\OO$: $$e_1=\matr{1}{\zero}{\zero}{0},\; e_2=\matr{0}{\zero}{\zero}{1},\; \uu_i=\matr{0}{\cc_i}{\zero}{0},\;\vv_i=\matr{0}{\zero}{\cc_i}{0}$$
for $i=1,2,3$. The unity of $\OO$ is denoted by $1_{\OO}=e_1+e_2$.  We identify octonions
$$\al 1_{\OO},\;\matr{0}{\uu}{\zero}{0},\; \matr{0}{\zero}{\vv}{0}$$
with $\al\in\FF$, $\uu,\vv\in\FF^3$, respectively. Note that $\uu_i \uu_j= (-1)^{\epsilon_{ij}}\vv_k$ and  $\vv_i \vv_j= (-1)^{\epsilon_{ji}}\uu_k$, where $\{i,j,k\}=\{1,2,3\}$ and $\epsilon_{ij}$ is the parity of permutation 
$\left(
\begin{array}{ccc}
\!1\! & \!2\! & \!3\! \\
\!k\! & \!i\! & \!j\! \\
\end{array}
\right)$.


The algebra $\OO$ is endowed with the linear involution
$$\ov{a}=\matr{\be}{-\uu}{-\vv}{\al},$$%
satisfying  $\ov{aa'}=\ov{a'}\ov{a}$, a {\it norm} $n(a)=a\ov{a}=\al\be-\uu\cdot \vv$, and a non-degenerate symmetric bilinear {\it form} $q(a,a')=n(a+a')-n(a)-n(a')=\al\be' + \al'\be - \uu\cdot \vv' - \uu'\cdot \vv$. Define the linear function {\it trace} by $\tr(a)=a + \ov{a} = \al+\be$. The subspace of traceless octonions is denoted by $\OO_0 = \{a\in\OO\, | \,\tr(a)=0\}$ and the affine variety of octonions with zero norm is denoted by $\OO_{\#}=\{a\in\OO\,|\,n(a)=0\}$. Notice that
\begin{eq}\label{eq1}
\tr(aa')=\tr(a'a) \text{ and } n(aa')=n(a)n(a').
\end{eq}%
The next quadratic equation holds:
\begin{eq}\label{eq2}
a^2 - \tr(a) a + n(a) = 0.
\end{eq}%
The algebra $\OO$ is a simple {\it alternative} algebra, i.e., the following identities hold for $a,b\in\OO$:
\begin{eq}\label{eq4}
a(ab)=(aa)b,\;\; (ba)a=b(aa).
\end{eq}%
Moreover, 
\begin{eq}\label{eq4b}
\ov{a}(ab) = n(a) b, \;\; (ba)\ov{a}=n(a) b.  
\end{eq}%

The following remark is well-known and can easily be proven.

\begin{remark}\label{remark_inv}
Given an $a\in \OO$, one of the following cases holds:
\begin{enumerate}
\item[$\bullet$] if $n(a)\neq0$, then there exist a unique $b,c\in\OO$ such that $ba=1_{\OO}$ and $ac=1_{\OO}$; moreover, in this case we have $b=c=\ov{a}/n(a)$ and we denote $a^{-1}=b=c$.  

\item[$\bullet$] if $n(a)=0$, then $a$ does not have a left inverse as well as a right inverse.
\end{enumerate}
\end{remark}

Equalities~(\ref{eq4b}) imply that for $a\not\in\OO_{\#}$ we have
\begin{eq}\label{eq_inv}
a^{-1}(ab) = b, \;\; (ba)a^{-1}=b.    
\end{eq}

\subsection{The group $\G$}\label{section_G2} The group $\Aut(\OO)$ of all automorphisms of the algebra $\OO$ is the exceptional simple group $\G=\G(\FF)$. The group $\G$ contains a Zariski closed subgroup $\SL_3=\SL_3(\FF)$. Namely, every $g\in\SL_3$ defines an automorphism of $\OO$ as follows:
$$a\to \matr{\al}{\uu g}{\vv g^{-T}}{\be},$$
where $g^{-T}$ stands for $(g^{-1})^T$ and $\uu,\vv\in\FF^3$ are considered as row vectors. For every $\uu,\vv\in \FF^3$ define $\de_1(\uu),\de_2(\vv)\in \Aut(\OO)$ as follows:
$$\de_1(\uu)(a')=\matr{\al' - \uu\cdot \vv'}{(\al'-\be' - \uu\cdot \vv')\uu + \uu'}{\vv' - \uu'\times \uu}{\be' + \uu\cdot\vv'},$$
$$\de_2(\vv)(a')=\matr{\al' + \uu'\cdot \vv}{\uu' + \vv'\times \vv}{(-\al'+\be' - \uu'\cdot \vv)\vv + \vv'}{\be' - \uu'\cdot\vv}.$$
The group $\G$ is generated by  $\SL_3$ and $\de_1(t\uu_i),\de_2(t\vv_i)$ for all $t\in\FF$ and $i=1,2,3$. As an example, by straightforward calculations we can see that
\begin{eq}
\hbar:\OO\to\OO, \text{ defined by }a \to \matr{\be}{-\vv}{-\uu}{\al},
\end{eq}%
belongs to $\G$.

The action of $\G$ on $\OO$ satisfies the following properties:
$$\ov{ga}=g\ov{a},\; \tr(ga)=\tr(a),\;n(ga)=n(a),\; q(ga,ga')=q(a,a'),$$%
for all $g \in \Aut(\OO)$ and $a,a' \in \OO$. Thus, $\G$ also acts on $\OO_0$ and $\OO_{\#}$. Moreover, the group $\G$  acts diagonally on the vector space $\OO^n=\OO\oplus \cdots\oplus\OO$ ($n$ times) by $g(a_1,\ldots,a_n)=(ga_1,\ldots,g a_n)$
for all $g\in \G$ and $a_1,\ldots,a_n\in\OO$.

\subsection{General octonions}\label{section_general_octonion}

An {\it octonion algebra} $\mathbf{C}$ (or a {\it Cayley algebra}) over an arbitrary field $\FF$ is a non-associative alternative unital algebra of dimension $8$, endowed with non-singular quadratic multiplicative form $n: \mathbf{C} \to \FF$, which is called the {\it norm}. As above, define over $\mathbf{C}$ the symmetric bilinear form $q(a,a')=n(a+a')-n(a) -n(a')$, the linear involution $\ov{a}=q(a,1_{\mathbf{C}})1_{\mathbf{C}} - a$, and the linear trace $\tr(a)=q(a,1_{\mathbf{C}})$. Then equality~(\ref{eq2}) also holds for every $a\in\mathbf{C}$, $\tr(a)1_{\mathbf{C}}=a+\ov{a}$ and $n(a)1_{\mathbf{C}}=a\ov{a}=\ov{a}a$. The following lemma is well-known.

\begin{lemma}\label{lemma_alg_closed}
Assume that the field $\FF$ is algebraically closed. Then any  octonion algebra $\mathbf{C}$ over $\FF$ is isomorphic to the split-octonion algebra $\OO_{\FF}$. 
\end{lemma}
\begin{proof}
If the norm $n$ on $\mathbf{C}$ is isotropic, then $\mathbf{C}$  is isomorphic to $\OO_{\FF}$. 

Assume that $n(a)\neq 0$ for every nonzero $a\in\mathbf{C}$. Consider some nonzero $a\in\mathbf{C}$ such that $a\neq \la 1_{\mathbf{C}}$ for every $\la\in\FF$. Denote $\al=n(a)\neq0$. Then 
$$n(a-\la 1_{\mathbf{C}})1_{\mathbf{C}} = (a-\la 1_{\mathbf{C}})(\ov{a}-\la 1_{\mathbf{C}}) = a \ov{a} -\la \tr(a) 1_{\mathbf{C}} + \la^2 1_{\mathbf{C}} = (\la^2 - \la \tr(a) + \al)1_{\mathbf{C}}.$$ 
Since $\FF$ is algebraically closed, then there exists $\la\in\FF$ such that $n(a-\la 1_{\mathbf{C}})=0$. Since $a-\la 1_{\mathbf{C}}$ is nonzero, we obtain a contradiction.
\end{proof}

\subsection{Notations}
We fix a binary relation $<$ on the field $\FF$ such that 
\begin{enumerate}
\item[$\bullet$] for each pair $\al,\be\in \FF$ with $\al\neq\be$ exactly one of $\al<\be$ or $\be<\al$ holds;

\item[$\bullet$] $\corr{\al}<0$ for every nonzero $\al\in\FF$.
\end{enumerate} 
Note that we do not assume that $<$ is transitive, and we do not assume compatibility with the field operations. Denote by $\FF^{\times}$ the set of all nonzero elements of $\FF$. We  write $a=(\al_1,\ldots,\al_8)$, where $\al_i\in\FF$, for an octonion
$$
a=\matr{\al_{1}}{(\al_{2}, \al_{3}, \al_{4})}{(\al_{5}, \al_{6}, \al_{7})}{\al_{8}}\in\OO.
$$%

\section{Classification of octonions}\label{section_def}

\subsection{Sets of octonions}

Introduce the following sets of diagonal octonions:
$$ {\rm(D)}\; \matr{\al_1}{\zero}{\zero}{\al_8},\qquad\quad 
{\rm(E)}\; \al_1 1_{\OO}, \qquad\quad
{\rm(F)}\; \matr{\al_1}{\zero}{\zero}{\al_8} \text{ with }\al_1\neq \al_8,
$$
where $\al_1,\al_8\in\FF$. Note that set (D) is the union of sets (E) and (F). Given $a\in\OO$, denote by $a^{\top}$ the transpose octonion of $a$, i.e.
$$a^{\top}=\matr{\al}{\vv}{\uu}{\be} \text{ for }a=\matr{\al}{\uu}{\vv}{\be}.$$
\noindent{}Consider the following sets of non-diagonal octonions:
\begin{enumerate}

\item[(K)] $\matr{\al_1}{(1,0,0)}{\zero}{\al_8}$,

\item[(L)] $\matr{\al_1}{(\al_2,0,0)}{\zero}{\al_8}$ with $\al_2\neq0$,

\item[(M)] $\matr{\al_1}{(0,1,0)}{\zero}{\al_8}$,

\item[(N)] $\matr{\al_1}{(1,0,0)}{(\al_5,0,0)}{\al_8}$  with $\al_5\neq0$,

\item[(${\mathrm{P}}$)] $\matr{\al_1}{(1,0,0)}{(0,1,0)}{\al_8}$,
\end{enumerate}
for $\al_1,\al_2,\al_5,\al_8\in\FF$.  Given some set (A) of octonions, we denote by
\begin{enumerate}
\item[($\mathrm{A}_0$)] the set of octonions $a\in\OO_0$ of set (A);

\item[($\mathrm{\ov{A}}$)] the set of octonions $a\in\OO$ of set (A) with $\al_1\leq\al_8$;

\item[($\mathrm{A}_1$)] the set of octonions $a\in\OO$ of set (A) with $\al_1=\al_8$;

\item[($\mathrm{A}^{\!\top}$)] the set of octonions $a^{\top}$ such that $a$ has set (A);

\item[($\mathrm{A}_0^{\!\top}$)] the set of octonions $a^{\top}$ such that $a\in\OO_0$ has set (A);

\item[($\mathrm{A}_1^{\!\top}$)] the set of octonions $a^{\top}$ such that $a$ has set (A) with $\al_1=\al_8$.
\end{enumerate}
\noindent{}Note that in case $\Char{\FF}=2$ sets ($\mathrm{A}_0$) and ($\mathrm{A}_1$) are the same.

\begin{prop}[Proposition 3.3 from~\cite{LZ_1}]\label{prop_O_canon}
 The following set is a minimal set of representatives of $\G$-orbits on $\OO$: 
\begin{enumerate}
\item[{\rm(E)}] $\al_1 1_{\OO}$,

\item[{\rm($\mathrm{\ov{F}}$)}] $\matr{\al_1}{\zero}{\zero}{\al_8}$ with $\al_1<\al_8$,

\item[{\rm($\mathrm{K}_1$)}] $\matr{\al_1}{(1,0,0)}{\zero}{\al_1}$,
\end{enumerate}
where $\al_1,\al_8\in\FF$. 
\end{prop}

\begin{theo}[Theorem 4.4 from~\cite{LZ_1}]\label{theo_O2_canon} The following set is a minimal set of representatives of $\G$-orbits on $\OO^2$: 
\begin{enumerate}
\item[{\rm (EE)}] $(\al_1 1_{\OO}, \be_1 1_{\OO})$, 

\item[{\rm (E$\mathrm{\ov{F}}$)}] $\left(\al_1 1_{\OO}, \matr{\be_1}{\zero}{\zero}{\be_8}\right)$ with $\be_1<\be_8$,

\item[{\rm($\rm E K_1$)}] $\left(\al_1 1_{\OO}, \matr{\be_1}{(1,0,0)}{\zero}{\be_1}\right)$,

\item[{\rm ($\mathrm{\ov{F}}$D)}] $\left(\matr{\al_1}{\zero}{\zero}{\al_8}, \matr{\be_1}{\zero}{\zero}{\be_8}\right)$ with $\al_1<\al_8$,

\item[{\rm($\mathrm{\ov{F}}$K)}] $\left(\matr{\al_1}{\zero}{\zero}{\al_8}, \matr{\be_1}{(1,0,0)}{\zero}{\be_8}\right)$ with $\al_1<\al_8$,

\item[{\rm($\mathrm{\ov{F} K^{\!\top}}$)}] $\left(\matr{\al_1}{\zero}{\zero}{\al_8}, \matr{\be_1}{\zero}{(1,0,0)}{\be_8}\right)$ with $\al_1<\al_8$,

\item[{\rm($\mathrm{\ov{F}}$N)}] $\left(\matr{\al_1}{\zero}{\zero}{\al_8}, \matr{\be_1}{(1,0,0)}{(\be_5,0,0)}{\be_8}\right)$ with $\al_1 <\al_8$ and $\be_5\neq0$,

\item[{\rm($\mathrm{\ov{F}}$P)}] $\left(\matr{\al_1}{\zero}{\zero}{\al_8}, \matr{\be_1}{(1,0,0)}{(0,1,0)}{\be_8}\right)$ with $\al_1< \al_8$,

\item[{\rm($\rm K_1E$)}] $\left(\matr{\al_1}{(1,0,0)}{\zero}{\al_1}, \be_1 1_{\OO}\right)$,

\item[{\rm($\rm K_1F)$}] $\left(\matr{\al_1}{(1,0,0)}{\zero}{\al_1}, \matr{\be_1}{\zero}{\zero}{\be_8}\right)$ with $\be_1\neq \be_8$,

\item[{\rm($\rm K_1L_1$)}] $\left(\matr{\al_1}{(1,0,0)}{\zero}{\al_1}, \matr{\be_1}{(\be_2,0,0)}{\zero}{\be_1}\right)$ with $\be_2\neq 0$,

\item[{\rm($\rm K_1\ov{L}^{\!\top}$)}] $\left(\matr{\al_1}{(1,0,0)}{\zero}{\al_1}, \matr{\be_1}{\zero}{(\be_5,0,0)}{\be_8}\right)$ with $\be_1\leq \be_8$ and $\be_5\neq 0$,

\item[{\rm($\rm K_1\ov{M}$)}] $\left(\matr{\al_1}{(1,0,0)}{\zero}{\al_1}, \matr{\be_1}{(0,1,0)}{\zero}{\be_8}\right)$ with $\be_1\leq \be_8$,

\item[{\rm($\rm K_1M_1^{\!\top}$)}] $\left(\matr{\al_1}{(1,0,0)}{\zero}{\al_1}, \matr{\be_1}{\zero}{(0,1,0)}{\be_1}\right)$,
\end{enumerate}
where $\al_1,\al_8,\be_1,\be_2,\be_5,\be_8\in\FF$.
\end{theo}

\subsection{Pairs of octonions with zero norm}\label{section_norm_zero}

The following remark is a consequence of Proposition~\ref{prop_O_canon}.

\begin{remark}\label{remark_O_nor0_canon}
 The following set is a minimal set of representatives of $\G$-orbits on $\OO_{\#}$:
\begin{eq}\label{eq_set1}
\{0,\; \al_1 e_1,\; \uu_1\,|\, \al_1\in \FF^{\times}\}.
\end{eq}
\end{remark}
\medskip

\begin{prop}\label{prop_O2_norm0_canon} The following set is a minimal set of representatives of $\G$-orbits on pairs of octonions with zero norm $\OO_{\#}^2$: 
\begin{enumerate}
\item[{\rm $\mathrm{(0)}$}] $(0, 0)$, 

\item[{\rm $\mathrm{(I)}$}] $(0, \be_1 e_1)$ with  $\be_1\in \FF^{\times}$,

\item[{\rm $\mathrm{(II)}$}] $(0, \uu_1)$, 

\item[{\rm $\mathrm{(III)}$}] $(\al_1 e_1, 0)$, 

\item[{\rm $\mathrm{(IV)}$}] $(\uu_1, 0)$, 

\item[$\mathrm{(V)}$] $\left(\al_1 e_1, \matr{\be_1}{\zero}{\zero}{\be_8}\right)$ with $\be_1\be_8=0$ and $(\be_1,\be_8)\neq(0,0)$,

\item[$\mathrm{(VI)}$] $\left(\al_1 e_1, \matr{\be_1}{(1,0,0)}{\zero}{\be_8}\right)$  with $\be_1\be_8=0$,

\item[$\mathrm{(VII)}$] $\left(\al_1 e_1, \matr{\be_1}{\zero}{(1,0,0)}{\be_8}\right)$ with $\be_1\be_8=0$,

\item[$\mathrm{(VIII)}$] $\left(\al_1 e_1, \matr{\be_1}{(1,0,0)}{(\be_1 \be_8,0,0)}{\be_8}\right)$ with $\be_1,\be_8\neq0$,

\item[$\mathrm{(IX)}$] $\left(\al_1 e_1, \matr{\be_1}{(1,0,0)}{(0,1,0)}{\be_8}\right)$ with $\be_1\be_8=0$,

\item[$\mathrm{(X)}$] $\left(\uu_1, \matr{\be_1}{\zero}{\zero}{\be_8}\right)$ with $\be_1\be_8=0$ and  $(\be_1,\be_8)\neq(0,0)$,

\item[$\mathrm{(XI)}$] $(\uu_1, \be_2\uu_1)$ with $\be_2\neq 0$,

\item[$\mathrm{(XII)}$] $\left(\uu_1, \matr{\be_1}{\zero}{(\be_5,0,0)}{0}\right)$ with  $\be_5\neq 0$,

\item[$\mathrm{(XIII)}$] $\left(\uu_1, \matr{\be_1}{(0,1,0)}{\zero}{0}\right)$,

\item[$\mathrm{(XIV)}$] $(\uu_1, \vv_2)$,
\end{enumerate}
where $\al_1\in\FF^{\times}$ and $\be_1,\be_2,\be_5,\be_8\in\FF$.
\end{prop}
\begin{proof} The claim follows from Theorem~\ref{theo_O2_canon} as the result of case by case consideration.
\end{proof}

\begin{defin}\label{def_canon}{}

\begin{enumerate}
\item[1.] The elements from set~(\ref{eq_set1}) are called {\it canonical non-invertible octonions}.

\item[2.] The pairs of elements from the formulation of Proposition~\ref{prop_O2_norm0_canon} are called {\it pairs of canonical non-invertible octonions}.
\end{enumerate}
\end{defin}

\section{Linear equation $ax=c$}\label{section_AXisC}

In this section we consider the linear equation
\begin{eq}\label{eq_AXisB}
ax = c,
\end{eq}%
where $a, c \in \OO$ are given constants and $x \in \OO$ is a variable. Clearly, if $a \notin \OO_{\#}$, then equation~(\ref{eq_AXisB}) has the unique solution $x = a^{-1}c$. The case $a \in \OO_{\#}$ is addressed in the following proposition.

\begin{prop}\label{prop_AXisB} 
Let $a \in \OO_{\#}$ be nonzero and $c \in \OO$. By acting with $\G$ on the equation $ax = c$, we may assume that $a$ is a canonical non-invertible octonion, as in Remark~\ref{remark_O_nor0_canon}. Denote by $X$ the solution set of the equation $ax = c$. Then $X$ is nonempty if and only if one of the following cases occurs for some $\gamma_i \in \FF$:
\begin{enumerate}
\item[1.] If $a = \alpha_1 e_1$ with $\alpha_1 \in \FF^{\times}$ and
$c = \begin{pmatrix}
\gamma_1 & (\gamma_2, \gamma_3, \gamma_4) \\
(0,0,0) & 0
\end{pmatrix}
$,
the set $X$ consists of all
\[
x = \begin{pmatrix}
\gamma_1 / \alpha_1 & (\gamma_2 / \alpha_1, \gamma_3 / \alpha_1, \gamma_4 / \alpha_1) \\
(x_5, x_6, x_7) & x_8
\end{pmatrix} \quad \text{for any } x_i \in \FF.
\]

\item[2.] If $a = \uu_1$ and 
$c = \begin{pmatrix}
\gamma_1 & (\gamma_2, 0, 0) \\
(0, \gamma_6, \gamma_7) & 0
\end{pmatrix},
$
the set $X$ consists of all
\[
x = \begin{pmatrix}
x_1 & (x_2, \gamma_7, -\gamma_6) \\
(\gamma_1, x_6, x_7) & \gamma_2
\end{pmatrix} \quad \text{for any } x_i \in \FF.
\]

\end{enumerate}

\noindent{}Moreover, the equation $ax = c$ has a solution if and only if $\overline{a} c = 0$.
\end{prop}
\begin{proof} In case $x$ is a solution for $ax=c$, equation~(\ref{eq4b}) implies that $0=n(a)x=\ov{a}(ax)=\ov{a}c$. 

Assume that $\ov{a}c=0$. We will show that in this case there is a solution for $ax=c$ and case 1 or 2 from the formulation of the \corr{proposition} holds. Denote octonions $c=(\ga_1,\ldots,\ga_8)$ and $x=(x_1,\ldots,x_8)$, where $\ga_j,x_j\in\FF$.
\medskip

\noindent{\bf 1.} Assume that $a=\al_1 e_1$ with $\al_1\in\FF^{\times}$.  Then equality $\ov{a}c=0$ implies that $\ga_5=\cdots=\ga_8=0$. Hence, $x$ is a solution for $ax=c$ if and only if $x_i=\ga_i/\al_1$ for $i=1,\ldots,4$.

\medskip
\noindent{\bf 2.} Assume that $a=\uu_1$.  Then equality $\ov{a}c=0$ implies that $\ga_3=\ga_4=\ga_5=\ga_8=0$. Hence, $x$ is a solution for $ax=c$ if and only if $x_3 = \ga_7$, $x_4 = -\ga_6$,  $x_5 = \ga_1$, $x_8 = \ga_2$.
\end{proof}

\begin{cor}\label{cor_AXisB}
Let $a \in \OO$ be nonzero and $c \in \OO$, and let $X \subset \OO$ denote the set of all solutions of the equation $ax = c$. Then
\begin{enumerate}
\item[(a)] $|X| = 1$ if and only if $a \notin \OO_{\#}$;

\item[(b)] $X \in \Omega_4$ if and only if $a \in \OO_{\#}$ and $\overline{a} c = 0$;

\item[(c)] $X$ is empty if and only if $a \in \OO_{\#}$ and $\overline{a} c \neq 0$.
\end{enumerate}
\end{cor}

\begin{proof}
If $a \notin \OO_{\#}$, then $a$ is invertible and $x = a^{-1}c$ is the unique solution of $ax = c$ by formula~(\ref{eq_inv}). 

Assume $a \in \OO_{\#}$. Proposition~\ref{prop_AXisB} implies that either $X \in \Omega_4$ or $X$ is empty. Since $X$ is nonempty if and only if $\overline{a} c = 0$, the proof is complete.
\end{proof}

\section{Linear equation $(ax)b=c$}\label{section_AXBisC}

In this section we consider the linear equation
\begin{eq}\label{eq_AXBisC}
(ax)b = c,
\end{eq}%
where $a, b, c \in \OO$ are given constants and $x \in \OO$ is a variable. Clearly, if $a \notin \OO_{\#}$ or $b \notin \OO_{\#}$, then solving equation~(\ref{eq_AXBisC}) can be reduced to solving an equation of the form $a'x = b'$ for some $a', b' \in \OO$ (see Corollary~\ref{cor_AXBisC} below for details). The case $a, b \in \OO_{\#}$ is addressed in the following theorem.

\begin{theo}\label{theo_AXBisC}  Let $a, b \in \OO_{\#}$ be nonzero and $c \in \OO$. By acting with $\G$ on the equation $(ax)b = c$, we may assume that $(a,b)$ is a canonical pair of non-invertible octonions as in Proposition~\ref{prop_O2_norm0_canon}. Denote by $X$ the solution set of the equation $(ax)b = c$. Then $X$ is nonempty if and only if one of the following cases occurs for some $\alpha_1 \in \FF^{\times}$ and $\beta_i, \gamma_i \in \FF$:
\begin{enumerate}
\item[$\mathrm{(V)}$] 1. For $(a,b)=(\al_1 e_1, \be_8 e_2)$, $c=\ga_2 \uu_1 + \ga_3 \uu_2 + \ga_4 \uu_3$ with $\be_8\neq0$, the set $X$ consists of all
$$x=\matr{x_1}{(\frac{\ga_2}{\al_1\be_8},\frac{\ga_3}{\al_1 \be_8},\frac{\ga_4}{\al_1 \be_8})}{(x_5,x_6,x_7)}{x_8} \text{ for any }x_i\in\FF.$$

\noindent{} 2. For $(a,b)=(\al_1 e_1, \be_1 e_1)$, $c=\ga_1 e_1$ with $\be_1\neq0$, the set $X$ consists of all
$$x=\matr{\frac{\ga_1}{\al_1\be_1}}{(x_2,x_3,x_4)}{(x_5,x_6,x_7)}{x_8} \text{ for any }x_i\in\FF.$$

\item[$\mathrm{(VI)}$] 1. For $(a,b)=(\al_1 e_1, \be_8 e_2 + \uu_1)$ and 
$$c=\matr{0}{(\ga_2, -\be_8 \ga_7, \be_8 \ga_6)}{(0,\ga_6,\ga_7)}{0}$$
with $\be_8\neq0$, the set $X$ consists of all
$$x=\matr{\frac{\ga_2}{\al_1} - \be_8 x_2}{(x_2, -\frac{\ga_7}{\al_1}, \frac{\ga_6}{\al_1})}{(x_5, x_6, x_7)}{x_8} \text{ for any }x_i\in\FF.$$

\noindent{} 2. For $(a,b)=(\al_1 e_1, \be_1 e_1 + \uu_1)$ and 
$$c=\matr{\be_1 \ga_2}{(\ga_2, 0, 0)}{(0,\ga_6, \ga_7)}{0},$$
the set $X$ consists of all
$$x=\matr{\frac{\ga_2}{\al_1}}{(x_2, -\frac{\ga_7}{\al_1}, \frac{\ga_6}{\al_1})}{(x_5, x_6, x_7)}{x_8} \text{ for any }x_i\in\FF.$$

\item[$\mathrm{(VII)}$] 1. For $(a,b)=(\al_1 e_1, \be_8  e_2 + \vv_1)$ and  $$c=\matr{\ga_1}{(\be_8\ga_1,\ga_3,\ga_4)}{(0,0,0)}{0}$$
with $\be_8\neq0$, the set $X$ consists of all
$$x=\matr{x_1}{(\frac{\ga_1}{\al_1}, \frac{\ga_3}{\al_1\be_8}, \frac{\ga_4}{\al_1\be_8})}{(x_5,x_6,x_7)}{x_8} \text{ for any }x_i\in\FF.$$

\noindent{} 2. For $(a,b)=(\al_1e_1, \be_1e_1 + \vv_1)$, $c=\ga_1e_1$, the set $X$ consists of all
$$x=\matr{x_1}{(\frac{\ga_1}{\al_1}-\be_1 x_1, x_3, x_4)}{(x_5, x_6, x_7)}{x_8} \text{ for any }x_i\in\FF.$$

\item[$\mathrm{(VIII)}$] For $(a,b)=(\al_1 e_1, \be_1 e_1 + \be_8 e_2 +\uu_1 + \be_1\be_8\vv_1)$ and 
$$c=\matr{\be_1\ga_2}{(\ga_2, -\be_8\ga_7, \be_8\ga_6)}{(0, \ga_6, \ga_7)}{0}$$
with $\be_1,\be_8\neq0$, the set $X$ consists of all
$$x=\matr{\frac{\ga_2}{\al_1}-\be_8 x_2}{(x_2, -\frac{\ga_7}{\al_1}, \frac{\ga_6}{\al_1})}{(x_5,x_6,x_7)}{x_8} \text{ for any }x_i\in\FF.$$

\item[$\mathrm{(IX)}$] 1. For $(a,b)=(\al_1 e_1, \be_8e_2 + \uu_1 + \vv_2)$ and $$c=\matr{\ga_1}{(\ga_2, \be_8\ga_1, \be_8\ga_6)}{(0, \ga_6, -\ga_1)}{0}$$ 
with $\be_8\neq0$, the set $X$ consists of all
$$x=\matr{\frac{\ga_2}{\al_1}-\be_8x_2}{(x_2, \frac{\ga_1}{\al_1}, \frac{\ga_6}{\al_1})}{(x_5, x_6, x_7)}{x_8} \text{ for any }x_i\in\FF.$$

\noindent{} 2. For $(a,b)=(\al_1e_1, \be_1e_1 + \uu_1 + \vv_2)$ and 
$$c=\matr{\ga_1}{(\ga_2, 0, 0)}{(0, \ga_6, -\ga_1 + \be_1 \ga_2)}{0},$$ 
the set $X$ consists of all
$$x=\matr{\frac{\ga_2}{\al_1}}{(x_2,\frac{\ga_1 - \be_1\ga_2}{\al_1},\frac{\ga_6}{\al_1})}{(x_5, x_6, x_7)}{x_8} \text{ for any }x_i\in\FF.$$

\item[$\mathrm{(X)}$] 1. For $(a,b)=(\uu_1, \be_8 e_2)$, $c=\ga_2 \uu_1$ with $\be_8\neq0$, the set $X$ consists of all
$$x=\matr{x_1}{(x_2,x_3,x_4)}{(x_5,x_6,x_7)}{\frac{\ga_2}{\be_8}} \text{ for any }x_i\in\FF.$$

\noindent{} 2. For $(a,b)=(\uu_1, \be_1 e_1)$, $c=\ga_1 e_1 + \ga_6\vv_2 +\ga_7\vv_3$ with $\be_1\neq0$, the set $X$ consists of all
$$x=\matr{x_1}{(x_2, \frac{\ga_7}{\be_1}, -\frac{\ga_6}{\be_1})}{(\frac{\ga_1}{\be_1},x_6,x_7)}{x_8} \text{ for any }x_i\in\FF.$$

\item[$\mathrm{(XI)}$] For $(a,b)=(\uu_1,\be_2 \uu_1)$, $c=\ga_2\uu_1$ with $\be_2\neq0$, the set $X$ consists of all
$$x=\matr{x_1}{(x_2,x_3,x_4)}{(\frac{\ga_2}{\be_2},x_6,x_7)}{x_8} \text{ for any }x_i\in\FF.$$

\item[$\mathrm{(XII)}$] For $(a,b)=(\uu_1, \be_1e_1 + \be_5 \vv_1)$ and $$c=\matr{\ga_1}{(0,\ga_3,\ga_4)}{(0,\frac{\be_1\ga_4}{\be_5}, -\frac{\be_1\ga_3}{\be_5})}{0}$$
with $\be_5\neq0$,  the set $X$ consists of all
$$x=\matr{x_1}{(x_2, -\frac{\ga_3}{\be_5}, -\frac{\ga_4}{\be_5})}{(x_5, x_6, x_7)}{\frac{\ga_1-\be_1x_5}{\be_5}} \text{ for any }x_i\in\FF.$$

\item[$\mathrm{(XIII)}$] For $(a,b)=(\uu_1, \be_1 e_1 + \uu_2)$ and 
$$c=\matr{\be_1 \ga_3}{(0, \ga_3, 0)}{(0, \be_1 \ga_8, \ga_7)}{\ga_8},$$ 
the set $X$ consists of all
$$x=\matr{x_1}{(x_2, x_3, -\ga_8)}{(\ga_3, x_6, x_7)}{\ga_7 - \be_1 x_3} \text{ for any }x_i\in\FF.$$

\item[$\mathrm{(XIV)}$] For $(a,b)=(\uu_1,\vv_2)$, $c=\ga_2 \uu_1$, we have 
$$x=\matr{x_1}{(x_2, \ga_2, x_4)}{(x_5, x_6, x_7)}{x_8} \text{ for any }x_i\in\FF.$$
\end{enumerate}
\end{theo}
\begin{proof} Denote octonions $c=(\ga_1,\ldots,\ga_8)$ and $x=(x_1,\ldots,x_8)$, where $\ga_j,x_j\in\FF$. Since $(a,b)$ is a canonical pair of non-invertible nonzero octonions from Proposition~\ref{prop_O2_norm0_canon}, then one of the following ten cases holds, where $\al_1\in\FF$ is always nonzero.

\begin{enumerate}
\item[$\mathrm{(V)}$] We have $(a,b)=(\al_1 e_1, \be_1 e_1 +\be_8 e_2)$ with $\be_1\be_8=0$ and $(\be_1,\be_8)\neq(0,0)$.  Then equation~(\ref{eq_AXBisC}) implies that $\ga_5=\cdots=\ga_8=0$. Since $\be_1\be_8=0$, we have that one of the next two cases holds. 
\begin{enumerate}
\item[1.] Assume $\be_1=0$ and $\be_8\neq0$. Hence, it follows from equation~(\ref{eq_AXBisC}) that $\ga_1 = 0$, $x_2 = \ga_2/\al_1 \be_8$,  $x_3 = \ga_3/\al_1 \be_8$,  $x_4 = \ga_4/\al_1 \be_8$.

\item[2.] Assume that $\be_1\neq0$ and $\be_8=0$.  Hence, it follows from equation~(\ref{eq_AXBisC}) that $\ga_2 = 0$, $\ga_3 = 0$, $\ga_4 = 0$, $x_1 = \ga_1/\al_1 
\be_1$.
\end{enumerate}

\medskip
\item[$\mathrm{(VI)}$] We have $(a,b)=(\al_1 e_1, \be_1 e_1 +\be_8 e_2 + \uu_1)$  with $\be_1\be_8=0$.  Since $\be_1\be_8=0$, we have that one of the next two cases holds. 
\begin{enumerate}
\item[1.] Assume $\be_1=0$ and $\be_8\neq0$. Then equation~(\ref{eq_AXBisC}) implies that $\ga_1 = 0$, $\ga_5 = 0$, $\ga_8 = 0$.  Hence, it follows from equation~(\ref{eq_AXBisC}) that $x_1 = (\ga_2/\al_1) - \be_8 x_2$, $x_3 = \ga_3/\al_1 \be_8$, $x_4 = \ga_4/\al_1 \be_8$. Finally,  equation~(\ref{eq_AXBisC}) implies that $\ga_4 = \be_8 \ga_6$ and  $\ga_3 = -\be_8 \ga_7$. 

\item[2.] Assume that $\be_8=0$.  Then equation~(\ref{eq_AXBisC}) implies that $\ga_3 = 0$, $\ga_4 = 0$, $\ga_5 = 0$, $\ga_8 = 0$.  Hence, it follows from equation~(\ref{eq_AXBisC}) that $x_1 = \ga_2/\al_1$, $x_3 = -\ga_7/\al_1$, $x_4 = \ga_6/\al_1$. Finally,  equation~(\ref{eq_AXBisC}) implies that $\ga_1 = \be_1 \ga_2$. 
\end{enumerate}

\medskip
\item[$\mathrm{(VII)}$] We have $(a,b)=(\al_1 e_1, \be_1 e_1 +\be_8 e_2 + \vv_1)$ with $\be_1\be_8=0$. Then equation~(\ref{eq_AXBisC}) implies that $\ga_5 = 0$, $\ga_6 = 0$, $\ga_7 = 0$, $\ga_8 = 0$.  Since $\be_1\be_8=0$, we have that one of the next two cases holds. 
\begin{enumerate}
\item[1.] Assume $\be_1=0$ and $\be_8\neq0$. Then equation~(\ref{eq_AXBisC}) implies that  $x_2 = \ga_2/\al_1 \be_8$, $x_3 = \ga_3/\al_1 \be_8$, $x_4 = \ga_4/\al_1 \be_8$.  Hence, it follows from equation~(\ref{eq_AXBisC}) that $\ga_2 = \be_8 \ga_1$. 

\item[2.] Assume that $\be_8=0$. Hence, it follows from equation~(\ref{eq_AXBisC}) that  $\ga_2 = 0$, $\ga_3 = 0$, $\ga_4 = 0$. Then equation~(\ref{eq_AXBisC}) implies that $x_2 = (
\ga_1/\al_1) - \be_1 x_1$. 
\end{enumerate}

\medskip
\item[$\mathrm{(VIII)}$] We have $(a,b)=\left(\al_1 e_1, \matr{\be_1}{(1,0,0)}{(\be_1 \be_8,0,0)}{\be_8}\right)$ with $\be_1,\be_8\neq0$. Then equation~(\ref{eq_AXBisC}) implies that $\ga_5 = 0$ and $\ga_8 = 0$. Hence, it follows from equation~(\ref{eq_AXBisC}) that $x_1 = (\ga_2/\al_1) - \be_8 x_2$, $x_3 = -\ga_7/\al_1$, $x_4 = \ga_6/\al_1$. Finally, it follows from equation~(\ref{eq_AXBisC}) that $\ga_1 = \be_1 \ga_2$, $\ga_3 = -\be_8 \ga_7$, $\ga_4 = \be_8 \ga_6$. 

\medskip
\item[$\mathrm{(IX)}$] We have $(a,b)=\left(\al_1 e_1, \matr{\be_1}{(1,0,0)}{(0,1,0)}{\be_8}\right)$ with $\be_1\be_8=0$. Then equation~(\ref{eq_AXBisC}) implies that $\ga_5 = 0$ and $\ga_8 = 0$.  Since $\be_1\be_8=0$, we have that one of the next two cases holds. 
\begin{enumerate}
\item[1.] Assume $\be_1=0$ and $\be_8\neq0$. Then equation~(\ref{eq_AXBisC}) implies that  $x_1 = (\ga_2/\al_1) - \be_8 x_2$, $x_3 = \ga_1/\al_1$, $x_4 = \ga_6/\al_1$.  Hence, it follows from equation~(\ref{eq_AXBisC}) that $\ga_3 = \be_8 \ga_1$, $\ga_4 = \be_8 \ga_6$, $\ga_7 = -\ga_1$. 

\item[2.] Assume that $\be_8=0$. Hence, it follows from equation~(\ref{eq_AXBisC}) that   $\ga_3 = 0$ and $\ga_4 = 0$. Then equation~(\ref{eq_AXBisC}) implies that $x_1 = \ga_2/\al_1$, $x_3 = -\ga_7/\al_1$ and  $x_4 = \ga_6/\al_1$. Finally, equation~(\ref{eq_AXBisC}) implies that $\ga_7 = -\ga_1 + \be_1 \ga_2$. 
\end{enumerate}

\medskip
\item[$\mathrm{(X)}$] We have $(a,b)=(\uu_1, \be_1 e_1 +\be_8 e_2)$ with $\be_1\be_8=0$ and  $(\be_1,\be_8)\neq(0,0)$.  Then equation~(\ref{eq_AXBisC}) implies that $\ga_3 = 0$, $\ga_4 = 0$, $\ga_5 = 0$, $\ga_8 = 0$.  Since $\be_1\be_8=0$, we have that one of the next two cases holds. 
\begin{enumerate}
\item[1.] Assume $\be_1=0$ and $\be_8\neq0$. Hence, it follows from equation~(\ref{eq_AXBisC}) that  $\ga_1 = 0$, $\ga_6 = 0$, $\ga_7 = 0$, $x_8 = \ga_2/\be_8$. 

\item[2.] Assume $\be_1\neq0$ and $\be_8=0$. Hence, it follows from equation~(\ref{eq_AXBisC}) that   $\ga_2 = 0$, $x_3 = \ga_7/\be_1$, $x_4 = -\ga_6/\be_1$, $x_5 = \ga_1/\be_1$. 
\end{enumerate}

\medskip
\item[$\mathrm{(XI)}$] We have $(a,b)=(\uu_1, \be_2\uu_1)$ with $\be_2\neq 0$.
Then equation~(\ref{eq_AXBisC}) implies that $\ga_1 = 0$, $\ga_3 = \cdots = \ga_8 = 0$, $x_5 = \ga_2/\be_2$. 

\medskip
\item[$\mathrm{(XII)}$] We have $(a,b)=(\uu_1, \be_1 e_1 + \be_5\vv_1)$ with  $\be_5\neq 0$.  Then equation~(\ref{eq_AXBisC}) implies that $\ga_2 = 0$, $\ga_5 = 0$, $\ga_8 = 0$. Hence, it follows from equation~(\ref{eq_AXBisC}) that $x_3 = -\ga_3/\be_5$, $x_4 = -\ga_4/\be_5$, $x_8 = (\ga_1 - \be_1 x_5)/\be_5$. Finally, it follows from equation~(\ref{eq_AXBisC}) that $\ga_6 = \be_1 \ga_4/\be_5$ and $\ga_7 = -\be_1 \ga_3/\be_5$. 

\medskip
\item[$\mathrm{(XIII)}$] We have $(a,b)=(\uu_1, \be_1e_1+\uu_2)$.  Then equation~(\ref{eq_AXBisC}) implies that $\ga_2 = 0$, $\ga_4 = 0$, $\ga_5 = 0$, $x_4 = -\ga_8$. Hence, it follows from equation~(\ref{eq_AXBisC}) that $x_5 = \ga_3$, $x_8 = \ga_7 - \be_1 x_3$. Finally, it follows from equation~(\ref{eq_AXBisC}) that  $\ga_1 = \be_1 \ga_3$ and $\ga_6 = \be_1 \ga_8$. 

\medskip
\item[$\mathrm{(XIV)}$] We have $(a,b)=(\uu_1, \vv_2)$. Then equation~(\ref{eq_AXBisC}) implies that $\ga_1 = 0$, $x_3 = \ga_2$, $\ga_3 = \cdots= \ga_8 = 0$.
\end{enumerate}
\end{proof}

\begin{cor}\label{cor_AXBisC}
Given nonzero $a, b \in \OO$ and $c \in \OO$, let $X \subset \OO$ be the set of all solutions of the equation $(ax)b = c$. Then $X$ satisfies one of the following conditions:
\begin{enumerate}
\item[(a)] $X$ is empty;

\item[(b)] $|X| = 1$;

\item[(c)] $X \in \Omega_r$ for some $r \in \{4,5,7\}$.
\end{enumerate}
Moreover,
\begin{enumerate}
\item[$\bullet$] each of the above cases occurs for some nonzero $a, b \in \OO$ and $c \in \OO$;

\item[$\bullet$] $|X| = 1$ if and only if $a, b \notin \OO_{\#}$;

\item[$\bullet$] $X \in \Omega_4$ if and only if either $a \notin \OO_{\#}$, $b \in \OO_{\#}$, $b\overline{c} = 0$, or $a \in \OO_{\#}$, $b \notin \OO_{\#}$, $\overline{a}(cb^{-1}) = 0$;

\item[$\bullet$] $X \in \Omega_r$ for $r \in \{5,7\}$ if and only if $a, b \in \OO_{\#}$ and $X$ is nonempty.
\end{enumerate}
\end{cor}
\begin{proof} Assume that $a,b\not\in\OO_{\#}$. Then by~(\ref{eq_inv}) we have that $x$ is a solution of~(\ref{eq_AXBisC}) if and only if $x=a^{-1}(cb^{-1})$.

Assume that $a\not\in\OO_{\#}$ and $b\in\OO_{\#}$. Then $x$ is a solution of~(\ref{eq_AXBisC}) if and only if $ax$ is a solution of the equation $yb=c$, or equivalently, $\ov{ax}$ is a solution of $\ov{b}z=\ov{c}$, where $z$ is a variable. By Corollary~\ref{cor_AXisB}, the set $Z\subset\OO$ of all solutions of $\ov{b}z=\ov{c}$ is either empty or $Z\in \Omega_4$. In the first case, $X$ is empty. In the second case, we consider a linear invertible map $f:\OO\to\OO$,  $z\to a^{-1}\ov{z}$; therefore, $X=f(Z) \in \Omega_4$. Moreover, by Corollary~\ref{cor_AXisB} the second case holds if and only if $b\ov{c}=0$. Thus we have proven that
\begin{eq}
\text{if }b\ov{c}=0, \text{ then } X\in\Omega_4; \text{ otherwise, }X \text{ is empty}. 
\end{eq}

Assume that $a\in\OO_{\#}$ and $b\not\in\OO_{\#}$.  Then $x$ is a solution of~(\ref{eq_AXBisC}) if and only if $ax=cb^{-1}$. Corollary~\ref{cor_AXisB} implies that
\begin{eq}
\text{if }\ov{a}(cb^{-1})=0, \text{ then } X\in\Omega_4; \text{ otherwise, }X \text{ is empty}. 
\end{eq}

If $a,b\in\OO_{\#}$, then Theorem~\ref{theo_AXBisC} implies that $X$ is empty or $X \in \Omega_r$ for $r\in\{5,7\}$. The required statement is proven.
\end{proof}

\begin{remark}\label{rem53}
If the equation $(ax)b = c$ has a solution for some $a,b \in \OO_{\#}$, $c \in \OO$, then necessarily $c\overline{b} = 0$. However, the converse does not hold in general.
\end{remark}

\begin{proof}
If $(ax)b = c$, then multiplying both sides on the right by $\overline{b}$ gives
\[
c\overline{b} = ((ax)b)\overline{b} = n(b)\, a x = 0,
\]
using equalities~(\ref{eq4b}).  
On the other hand, for $a = e_1$, $b = \uu_1$, and $c = e_2$, we have $c\overline{b} = 0$, but the equation $(ax)b = c$ has no solution.
\end{proof}

\section{Linear equation $a(bx)=c$}\label{section_ABXisC}

In this section we consider the linear equation
\begin{eq}\label{eq_ABXisC}
a(bx)=c,
\end{eq}%
where $a, b, c \in \OO$ are given constants and $x \in \OO$ is a variable. Clearly, if $a \notin \OO_{\#}$ or $b \notin \OO_{\#}$, then solving equation~\eqref{eq_ABXisC} can be reduced to solving an equation of the form $a'x = b'$ for some $a', b' \in \OO$ (see Corollary~\ref{cor_ABXisC} below for details). The case $a, b \in \OO_{\#}$ is addressed in the following theorem.

\begin{theo}\label{theo_ABXisC}  Given a nonzero $a,b\in\OO_{\#}$ and $c\in\OO$, acting by $\G$ on the equation $a(bx)=c$  we can assume that $(a,b)$ is a canonical pair of non-invertible octonions from Proposition~\ref{prop_O2_norm0_canon}. Denote by $X$ the solution set of the equation $a(bx)=c$. Then $X$ is nonempty if and only if one of the following cases holds for some $\al_1\in\FF^{\times}$, $\be_i,\ga_i\in\FF$:
\begin{enumerate}
\item[$\mathrm{(V)}$] 1. For $(a,b)=(\al_1 e_1, \be_8 e_2)$, $c=0$ with $\be_8\neq0$, the set $X$ consists of all
$$x=\matr{x_1}{(x_2,x_3,x_4)}{(x_5,x_6,x_7)}{x_8} \text{ for any }x_i\in\FF.$$

\noindent{} 2. For $(a,b)=(\al_1 e_1, \be_1 e_1)$ and 
$$c=\matr{\ga_1}{(\ga_2,\ga_3,\ga_4)}{(0,0,0)}{0}$$ 
with $\be_1\neq0$, the set $X$ consists of all
$$x=\matr{\frac{\ga_1}{\al_1\be_1}}{(\frac{\ga_2}{\al_1\be_1}, \frac{\ga_3}{\al_1\be_1}, \frac{\ga_4}{\al_1\be_1})}{(x_5,x_6,x_7)}{x_8} \text{ for any }x_i\in\FF.$$

\item[$\mathrm{(VI)}$] 1. For $(a,b)=(\al_1 e_1, \be_8 e_2 + \uu_1)$, $c=\ga_1 e_1 + \ga_2 \uu_1$, the set $X$ consists of all 
$$x=\matr{x_1}{(x_2,x_3,x_4)}{(\frac{\ga_1}{\al_1},x_6,x_7)}{\frac{\ga_2}{\al_1}} \text{ for any }x_i\in\FF.$$

\noindent{} 2. For $(a,b)=(\al_1 e_1, \be_1 e_1 + \uu_1)$ and 
$$c=\matr{\ga_1}{(\ga_2, \ga_3, \ga_4)}{(0, 0, 0)}{0}$$
with $\be_1\neq0$, the set $X$ consists of all 
$$x=\matr{x_1}{(x_2,\frac{\ga_3}{\al_1\be_1} ,\frac{\ga_4}{\al_1\be_1} )}{(\frac{\ga_1}{\al_1} - \be_1 x_1, x_6,x_7)}{\frac{\ga_2}{\al_1} - \be_1 x_2} \text{ for any }x_i\in\FF.$$

\item[$\mathrm{(VII)}$] 1. For $(a,b)=(\al_1 e_1, \be_8  e_2 + \vv_1)$, $c=\ga_3\uu_2 + \ga_4\uu_3$
the set $X$ consists of all
$$x=\matr{x_1}{(x_2,x_3,x_4)}{(x_5, -\frac{\ga_4}{\al_1}, \frac{\ga_3}{\al_1})}{x_8} \text{ for any }x_i\in\FF.$$

\noindent{} 2. For $(a,b)=(\al_1e_1, \be_1e_1 + \vv_1)$ and
$$c=\matr{\ga_1}{(\ga_2, \ga_3, \ga_4)}{(0, 0, 0)}{0}$$
with $\be_1\neq0$, the set $X$ consists of all
$$x=\matr{\frac{\ga_1}{\al_1\be_1}}{(\frac{\ga_2}{\al_1\be_1},x_3,x_4)}{(x_5, -\frac{\ga_4}{\al_1} + \be_1x_4, \frac{\ga_3}{\al_1} - \be_1x_3)}{x_8} \text{ for any }x_i\in\FF.$$

\item[$\mathrm{(VIII)}$] For $(a,b)=(\al_1 e_1, \be_1 e_1 + \be_8 e_2 +\uu_1 + \be_1\be_8\vv_1)$ and 
$$c=\matr{\ga_1}{(\ga_2, \ga_3, \ga_4)}{(0, 0, 0)}{0}$$ 
with $\be_1,\be_8\neq0$, the set $X$ consists of all 
$$x=\matr{x_1}{(x_2,\frac{\ga_3}{\al_1\be_1} - \be_8 x_7, \frac{\ga_4}{\al_1\be_1} + \be_8 x_6)}{(\frac{\ga_1}{\al_1} - \be_1 x_1, x_6, x_7)}{\frac{\ga_2}{\al_1} - \be_1 x_2} \text{ for any }x_i\in\FF.$$

\item[$\mathrm{(IX)}$] 1. For $(a,b)=(\al_1 e_1, \be_8e_2 + \uu_1 + \vv_2)$ and $$c=\matr{\ga_1}{(\ga_2, 0, \ga_1)}{(0, 0, 0)}{0},$$ 
the set $X$ consists of all  
$$x=\matr{x_1}{(x_2,x_3,x_4)}{(\frac{\ga_1}{\al_1},x_6,x_7)}{\frac{\ga_2}{\al_1} + x_7} \text{ for any }x_i\in\FF.$$

\noindent{} 2. For $(a,b)=(\al_1e_1, \be_1e_1 + \uu_1 + \vv_2)$ and 
$$c=\matr{\ga_1}{(\ga_2, \ga_3, \ga_4)}{(0, 0, 0)}{0}$$
with $\be_1\neq0$, the set $X$ consists of all 
$$x=\matr{x_1}{(x_2, \frac{\ga_3}{\al_1\be_1}, \frac{\ga_4 - \ga_1}{\al_1\be_1} + x_1)}{(\frac{\ga_1}{\al_1} - \be_1x_1, x_6, x_7)}{\frac{\corr{\ga_2}}{\al_1} - \be_1x_2 + x_7} \text{ for any }x_i\in\FF.$$

\item[$\mathrm{(X)}$] 1. For $(a,b)=(\uu_1, \be_8 e_2)$, $c=\ga_1 e_1 + \ga_2 \uu_1$ with $\be_8\neq0$, the set $X$ consists of all  
$$x=\matr{x_1}{(x_2,x_3,x_4)}{(\frac{\ga_1}{\be_8},x_6,x_7)}{\frac{\ga_2}{\be_8}} \text{ for any }x_i\in\FF.$$

\noindent{} 2. For $(a,b)=(\uu_1, \be_1 e_1)$, $c=\ga_6\vv_2 +\ga_7\vv_3$ with $\be_1\neq0$, the set $X$ consists of all  
$$x=\matr{x_1}{(x_2,\frac{\ga_7}{\be_1},-\frac{\ga_6}{\be_1})}{(x_5,x_6,x_7)}{x_8} \text{ for any }x_i\in\FF.$$

\item[$\mathrm{(XI)}$] For $(a,b)=(\uu_1,\be_2 \uu_1)$, $c=0$ with $\be_2\neq0$, the set $X$ consists of all  
$$x=\matr{x_1}{(x_2,x_3,x_4)}{(x_5,x_6,x_7)}{x_8} \text{ for any }x_i\in\FF.$$

\item[$\mathrm{(XII)}$] For $(a,b)=(\uu_1, \be_1e_1 + \be_5 \vv_1)$ and $$c=\matr{\ga_1}{(\ga_2, 0, 0)}{(0, \ga_6, \ga_7)}{0}$$
with $\be_5\neq0$, the set $X$ consists of all 
$$x=\matr{\frac{\ga_1}{\be_5}}{(\frac{\ga_2}{\be_5}, x_3,x_4)}{(x_5, \frac{\ga_6 + \be_1x_4}{\be_5} \corr{,}  \frac{\ga_7 - \be_1x_3}{\be_5})}{x_8} \text{ for any }x_i\in\FF.$$

\item[$\mathrm{(XIII)}$] For $(a,b)=(\uu_1, \be_1 e_1 + \uu_2)$ and 
$$c=\matr{\ga_1}{(0, 0, 0)}{(0, -\be_1\ga_1, \ga_7)}{0},$$ 
the set $X$ consists of all 
$$x=\matr{x_1}{(x_2,x_3,\ga_1)}{(x_5,x_6,x_7)}{\ga_7 - \be_1 x_3} \text{ for any }x_i\in\FF.$$

\item[$\mathrm{(XIV)}$] For $(a,b)=(\uu_1,\vv_2)$, $c=\ga_2 \uu_1 + \ga_6\vv_2$, we have  
$$x=\matr{x_1}{(x_2, \ga_2, x_4)}{(-\ga_6, x_6, x_7)}{x_8} \text{ for any }x_i\in\FF.$$
\end{enumerate}
\end{theo}
\begin{proof} Denote octonions $c=(\ga_1,\ldots,\ga_8)$ and $x=(x_1,\ldots,x_8)$, where $\ga_j,x_j\in\FF$. Since $(a,b)$ is a canonical pair of non-invertible nonzero octonions from Proposition~\ref{prop_O2_norm0_canon}, then one of the following ten cases holds, where $\al_1\in\FF$ is always nonzero.

\begin{enumerate}
\item[$\mathrm{(V)}$] $(a,b)=(\al_1 e_1, \be_1 e_1 +\be_8 e_2)$ with $\be_1\be_8=0$ and $(\be_1,\be_8)\neq(0,0)$. Then equation~(\ref{eq_ABXisC}) implies that  $\ga_5 = \cdots = \ga_8 = 0$. Since $\be_1\be_8=0$, we have that one of the next two cases holds. 
\begin{enumerate}
\item[1.] Assume $\be_1=0$ and $\be_8\neq0$. Hence, it follows from equation~(\ref{eq_ABXisC}) that $\ga_1 = \cdots =\ga_4 = 0$.

\item[2.] Assume that $\be_1\neq0$ and $\be_8=0$.  Hence, it follows from equation~(\ref{eq_ABXisC}) that $x_1 = \ga_1/\al_1 \be_1$, $x_2 = \ga_2/\al_1 \be_1$, $x_3 = \ga_3/\al_1 \be_1$, $x_4 = \ga_4/\al_1 \be_1$.
\end{enumerate}

\medskip
\item[$\mathrm{(VI)}$] $(a,b)=(\al_1 e_1, \be_1 e_1 +\be_8 e_2 + \uu_1)$  with $\be_1\be_8=0$.  Then equation~(\ref{eq_ABXisC}) implies that $\ga_5 = \cdots = \ga_8 = 0$. Since $\be_1\be_8=0$, we have that one of the next two cases holds. 
\begin{enumerate}
\item[1.] Assume $\be_1=0$. Hence, it follows from equation~(\ref{eq_ABXisC}) that $x_5 = \ga_1/\al_1$, $x_8 = \ga_2/\al_1$, $\ga_3 = 0$, $\ga_4 = 0$.

\item[2.] Assume that $\be_1\neq0$ and $\be_8=0$.  Hence, it follows from equation~(\ref{eq_ABXisC}) that $x_5 = (\ga_1/\al_1) - \be_1 x_1$, $x_8 = (\ga_2/\al_1) - \be_1 x_2$, $x_3 = \ga_3/\al_1 \be_1$,  $x_4 = \ga_4/\al_1 \be_1$.
\end{enumerate}

\medskip
\item[$\mathrm{(VII)}$] $(a,b)=(\al_1 e_1, \be_1 e_1 +\be_8 e_2 + \vv_1)$ with $\be_1\be_8=0$.  Then equation~(\ref{eq_ABXisC}) implies that $\ga_5 = \cdots =\ga_8 = 0$. Since $\be_1\be_8=0$, we have that one of the next two cases holds. 
\begin{enumerate}
\item[1.] Assume $\be_1=0$. Hence, it follows from equation~(\ref{eq_ABXisC}) that $\ga_1 = 0$, $\ga_2 = 0$, $x_6 = -\ga_4/\al_1$, $x_7 = \ga_3/\al_1$.

\item[2.] Assume that $\be_1\neq0$ and $\be_8=0$.  Hence, it follows from equation~(\ref{eq_ABXisC}) that $x_1 = \ga_1/\al_1 \be_1$, $x_2 = \ga_2/\al_1 \be_1$, $x_7 = (\ga_3/\al_1) - \be_1 x_3$, $x_6 = (\al_1 \be_1 x_4 - \ga_4) / \al_1$.
\end{enumerate}

\medskip
\item[$\mathrm{(VIII)}$] $(a,b)=\left(\al_1 e_1, \matr{\be_1}{(1,0,0)}{(\be_1 \be_8,0,0)}{\be_8}\right)$ with $\be_1,\be_8\neq0$.  Then equation~(\ref{eq_ABXisC}) implies that $\ga_5 = \cdots = \ga_8 = 0$.  Hence, it follows from equation~(\ref{eq_ABXisC}) that $x_3 = (\ga_3/\al_1 \be_1) - \be_8 x_7$, $x_4 = (\ga_4/\al_1 \be_1) + \be_8 x_6$, $x_5 = (\ga_1/\al_1) - \be_1 x_1$, $x_8 = (\ga_2/\al_1) - \be_1 x_2$. 

\medskip
\item[$\mathrm{(IX)}$] $(a,b)=\left(\al_1 e_1, \matr{\be_1}{(1,0,0)}{(0,1,0)}{\be_8}\right)$ with $\be_1\be_8=0$.  Then equation~(\ref{eq_ABXisC}) implies that $\ga_5 = \cdots = \ga_8 = 0$. Since $\be_1\be_8=0$, we have that one of the next two cases holds. 
\begin{enumerate}
\item[1.] Assume $\be_1=0$. Hence, it follows from equation~(\ref{eq_ABXisC}) that $\ga_3 = 0$, $x_5 = \ga_1/\al_1$, $x_8 = (\ga_2/\al_1) + x_7$. Therefore, $\ga_4 = \ga_1$ by equation~(\ref{eq_ABXisC}).

\item[2.] Assume that $\be_1\neq0$ and $\be_8=0$.  Hence, it follows from equation~(\ref{eq_ABXisC}) that $x_3 = \ga_3/\al_1 \be_1$, $x_5 = (\ga_1/\al_1) - \be_1 x_1$, $x_8 = (\ga_2/\al_1) + x_7 - \be_1 x_2$. Therefore, equation~(\ref{eq_ABXisC}) implies that
$$x_4 = \frac{\ga_4 - \ga_1}{\al_1 \be_1}  + x_1.$$  
\end{enumerate}

\medskip
\item[$\mathrm{(X)}$] $(a,b)=(\uu_1, \be_1 e_1 +\be_8 e_2)$ with $\be_1\be_8=0$ and  $(\be_1,\be_8)\neq(0,0)$.  Then equation~(\ref{eq_ABXisC}) implies that $\ga_3 = 0$, $\ga_4 = 0$, $\ga_5 = 0$, $\ga_8 = 0$. Since $\be_1\be_8=0$, we have that one of the next two cases holds. 
\begin{enumerate}
\item[1.] Assume $\be_1=0$ and $\be_8\neq0$. Hence, it follows from equation~(\ref{eq_ABXisC}) that $\ga_6 = 0$, $\ga_7 = 0$, $x_5 = \ga_1/\be_8$, $x_8 = \ga_2/\be_8$.

\item[2.] Assume that $\be_1\neq0$ and $\be_8=0$.  Hence, it follows from equation~(\ref{eq_ABXisC}) that $\ga_1 = 0$, $\ga_2 = 0$, $x_3 = \ga_7/\be_1$, $x_4 = -\ga_6/\be_1$.
\end{enumerate}

\medskip
\item[$\mathrm{(XI)}$] $(a,b)=(\uu_1, \be_2\uu_1)$ with $\be_2\neq 0$.  Then equation~(\ref{eq_ABXisC}) implies that $c=0$ and $x$ is an arbitrary. 

\medskip
\item[$\mathrm{(XII)}$] $(a,b)=(\uu_1, \be_1 e_1 + \be_5\vv_1)$ with  $\be_5\neq 0$.  Then equation~(\ref{eq_ABXisC}) implies that $\ga_3 = 0$, $\ga_4 = 0$, $\ga_5 = 0$, $\ga_8 = 0$.  Hence, it follows from equation~(\ref{eq_ABXisC}) that $x_1 = \ga_1/\be_5$, $x_2 = \ga_2/\be_5$, $x_6 = (\ga_6 + \be_1 x_4)/\be_5$, $x_7 = (\ga_7 - \be_1 x_3)/\be_5$. 

\medskip
\item[$\mathrm{(XIII)}$] $(a,b)=(\uu_1, \be_1e_1+\uu_2)$.  Then equation~(\ref{eq_ABXisC}) implies that $\ga_2 = \cdots = \ga_5 = 0$, $\ga_8 = 0$.  Hence, it follows from equation~(\ref{eq_ABXisC}) that $x_4 = \ga_1$, $x_8 = \ga_7 - \be_1 x_3$. Finally,  $\ga_6 = -\be_1 \ga_1$  by equation~(\ref{eq_ABXisC}).

\medskip
\item[$\mathrm{(XIV)}$] $(a,b)=(\uu_1, \vv_2)$.  Then equation~(\ref{eq_ABXisC}) implies that $\ga_1 = 0$, $\ga_3 = 0$, $\ga_4 = 0$, $\ga_5 = 0$, $\ga_7 = 0$, $\ga_8 = 0$, $x_3 = \ga_2$, $x_5 = -\ga_6$. 
\end{enumerate}
\end{proof}

\begin{cor}\label{cor_ABXisC}
Given nonzero $a,b \in \OO$ and $c \in \OO$, let $X \subset \OO$ be the set of all solutions of the equation $a(bx) = c$. Then $X$ satisfies one of the following conditions:
\begin{enumerate}
\item[(a)] $X$ is empty;

\item[(b)] $|X| = 1$;

\item[(c)] $X \in \Omega_r$ for some $r \in \{4,6\}$;

\item[(d)] $X = \OO$.
\end{enumerate}

Moreover,
\begin{enumerate}
\item[$\bullet$] each of the above cases occurs for some nonzero $a,b \in \OO$ and $c \in \OO$;

\item[$\bullet$] $|X| = 1$ if and only if $a,b \notin \OO_{\#}$;

\item[$\bullet$] $X = \OO$ if and only if $b = \xi\, \overline{a} \in \OO_{\#}$ for some $\xi \in \FF^{\times}$ and $c = 0$.
\end{enumerate}
\end{cor}
\begin{proof} Assume that $a,b\not\in\OO_{\#}$. Then by~(\ref{eq_inv}) we have that $x$ is a solution of equation~(\ref{eq_ABXisC}) if and only if $x=b^{-1}(a^{-1}c)$.

Assume that $a\not\in\OO_{\#}$ and $b\in\OO_{\#}$.  Then $x$ is a solution of~(\ref{eq_ABXisC}) if and only if $bx=a^{-1}c$. Corollary~\ref{cor_AXisB} implies that either $X$ is empty or $X\in\Omega_4$.

Assume that $a\in\OO_{\#}$ and $b\not\in\OO_{\#}$.  Then $x$ is a solution of~(\ref{eq_ABXisC}) if and only if $bx$ is a solution of the equation $ay=c$. By Corollary~\ref{cor_AXisB}, the set $Y\subset \OO$ of all solutions of the equation $ay=c$ is either empty or $Y\in \Omega_4$. In the first case, $X$ is empty. In the second case, $X=b^{-1}Y \in \Omega_4$, since the multiplication by $b^{-1}$ is the invertible linear map $\OO\to\OO$.  

If $a,b\in\OO_{\#}$, then Theorem~\ref{theo_ABXisC} implies that either $X$ is empty, or $X \in \Omega_r$ for $r\in\{4,6\}$, or $X=\OO$. Assume that $X=\OO$. Then it follows from Theorem~\ref{theo_ABXisC} that the $G_2$-orbit of $(a,b)$ contains $(\al e_1,\be e_2)$ or $(\uu_1,\be \uu_1)$ for some $\al,\be\in\FF^{\times}$ and $c=0$; in particular, $b=\xi\,\ov{a}$ for some $\xi\in\FF^{\times}$. Assume that $b=\xi\,\ov{a}\in\OO_{\#}$ for some $\xi\in\FF^{\times}$ and $c=0$. Then for every $x\in\OO$ we have $a(bx)=\xi a(\ov{a}x)=\xi\,n(a)x=0$ by equalities~(\ref{eq4b}), i.e., $X=\OO$. The required statement is proven.
\end{proof}

\begin{remark}\label{rem63}
If $a(bx)=c$ has a solution for some $a,b \in \OO_{\#}$ and $c \in \OO$, then $\overline{a}c = 0$. However, the converse does not hold in general.
\end{remark}

\begin{proof}
If $a(bx)=c$, then $\overline{a}c = n(a)\, bx = 0$ by equalities~(\ref{eq4b}). On the other hand, for $a = e_1$, $b = e_2$, and $c = e_1$, we have $\overline{a}c = 0$, but the equation $a(bx)=c$ has no solution.
\end{proof}

\section{Linear monomial equations}\label{section_monom}

We write $a_1 \mult \cdots \mult a_m$ for an arbitrary non-associative product of $a_1,\ldots,a_m \in \OO$.  
A {\it multilinear monomial} $w = w(a_1,\ldots,a_m)$ is a product of the form
$a_{\sigma(1)} \mult \cdots \mult a_{\sigma(m)}$
for some permutation $\sigma \in \Sym_m$.  
A {\it linear monomial equation} over the octonions is an equation
$w(a_1,\ldots,a_m,x) = c$,
where $w(a_1,\ldots,a_m,x)$ is a multilinear monomial, $x \in \OO$ is a variable, and $c \in \OO$ is a constant.  
For example, the equations $(ax)b = c$ and $a(bx) = c$ are linear monomial equations.

\begin{cor}\label{cor_minom_eq}
Assume that $a_1,\ldots,a_m \in \OO$ are nonzero and $c \in \OO$.  
Let $X \subset \OO$ be the set of all solutions of a linear monomial equation
\[
w(a_1,\ldots,a_m,x)=c.
\]
Then $X$ satisfies one of the following conditions:
\begin{enumerate}
\item[(a)] $X$ is empty;

\item[(b)] $|X|=1$;

\item[(c)] $X \in \Omega_r$ for some $4 \le r \le 7$;

\item[(d)] $X = \OO$.
\end{enumerate}

Moreover,
\begin{enumerate}
\item[$\bullet$] each of the above cases occurs for some linear monomial equation;

\item[$\bullet$] $|X|=1$ if and only if $a_1,\ldots,a_m \notin \OO_{\#}$.
\end{enumerate}
\end{cor}
\begin{proof} By Corollaries~\ref{cor_AXBisC} and~\ref{cor_ABXisC} each of the cases (a)--(d) from the formulation of the corollary hold for some $w(a_1,a_2,x)\in\{ (a_1x)a_2,\; a_1(a_2 x)\}$. The rest of the proof we split into two parts.

\medskip
\noindent{\bf 1.} If  $a_1,\ldots,a_m\not\in\OO_{\#}$, then applying equalities~(\ref{eq_inv}) $m$ times, we obtain that there exists a unique solution of the equation $w(a_1,\ldots,a_m,x)=c$.

\medskip
\noindent{\bf 2.} Assume that $a_i\in\OO_{\#}$ for some $1\leq i\leq m$ and $X$ is not empty, i.e., there exists $x_0\in\OO$ such that $w(a_1,\ldots,a_m,x_0)=c$. 

If $w(a_1,\ldots,a_m,x)=\cdots \mult (a_i a_j) \mult \cdots$ for some $1\leq i,j\leq m$, then we can denote $a'=a_i a_j$ and consider $w(a_1,\ldots,a_m,x)$ as a multilinear monomial in $\{x,a',a_r\,|\,1\leq r\leq m,\;r\neq i,j\}$; i.e.,  we diminish $m$ by one. \corr{Note that we can assume that $a'\neq0$, since otherwise $X=\emptyset$ in case $c\neq0$ and $X=\OO$ in case $c=0$.} After repeating this procedure several times, we can assume that 
\begin{eq}\label{eq_cor_cond}
w(x)\neq \cdots \mult (a_i  a_j) \mult \cdots \qquad \text{for all} \quad 1\leq i,j\leq m.
\end{eq}

Since condition~(\ref{eq_cor_cond}) holds, we have
\begin{align} 
w(a_1,\ldots,a_m,x) = a_{i_1} \mult \cdots  \mult a_{i_{r-1}} \mult (a_{i_r} v(x)) \mult a_{i_{r+1}} \mult \cdots  \mult a_{i_s}&\text{ or } \label{eq_canon1} \\
w(a_1,\ldots,a_m,x) = a_{i_1} \mult \cdots \mult a_{i_{r-1}}  \mult (v(x) a_{i_r}) \mult a_{i_{r+1}} \mult \cdots  \mult a_{i_s}&, \label{eq_canon2}
\end{align}%
where
\begin{enumerate}
\item[$\bullet$] $v(x):=v(a_{j_1},\ldots,a_{j_t},x)$ is a multilinear monomial for some $a_{j_1},\ldots,a_{j_t}\not\in\OO_{\#}$ with $t\geq 0$; 

\item[$\bullet$] $\{i_1,\ldots, i_s, j_1,\ldots, j_t\} = \{1,\ldots,m\}$ and $s+t=m$, $1\leq r\leq s$;

\item[$\bullet$] $a_{i_r}\in \OO_{\#}$.
\end{enumerate}

Denote by $X_1\subset \OO$ the set of solutions of the equation $a_{i_r} v(x) = a_{i_r} v(x_0)$ and denote by $X_2\subset \OO$ the set of solutions of the equation $v(x) a_{i_r}  =  v(x_0) a_{i_r}$. Note that  $X_1\subset X$ or $X_2\subset X$. Denote by $Y_1\subset \OO$ the set of solutions of the equation $a_{i_r} y = a_{i_r} v(x_0)$ and denote by $Y_2\subset \OO$ the set of solutions of the equation $y a_{i_r}  =  v(x_0) a_{i_r}$, where $y$ is a variable. Since $a_{j_1},\ldots,a_{j_t}$ are invertible,  equalities~(\ref{eq_inv}) imply that there are linear invertible maps $f_l: \OO \to \OO$ such that $f_l(Y_l) = X_l$, where $l=1,2$. One of the following two cases holds:

\begin{enumerate}
\item[$\bullet$] Assume that equality~(\ref{eq_canon1}) holds.  Then $x_0 \in X_1$ and $Y_1$ is also not empty. Corollary~\ref{cor_AXisB} implies that $Y_1\in \Omega_4$, since $a_{i_r}\in \OO_{\#}$. Applying the invertible linear map $f_1$ to $Y_1$, we can see that $X_1\in \Omega_4$. 

\item[$\bullet$] Assume that equality~(\ref{eq_canon2}) holds.  Then $x_0 \in X_2$ and $Y_2$ is also not empty. Applying the involution $\ov{\stackrel{\;\;}{\;\;}}:\OO\to \OO$ to the equation $y a_{i_r}  =  v(x_0) a_{i_r}$, we also obtain that  $Y_2\in \Omega_4$. Applying the invertible linear map $f_2$ to $Y_2$, we can see that $X_2\in \Omega_4$.
\end{enumerate}

Obviously, $X \in \Omega_r$ for some $0\leq r\leq 8$.  Since  $X_l\in \Omega_4$  and $X_l\subset X$ for some $l=1,2$, we obtain that $r\geq4$. The required statement is proven.
\end{proof}

\begin{cor}\label{cor_invertible}  If a linear monomial equation $w(a_1,\ldots,a_m,x)=c$ with nonzero $a_1,\ldots,a_m,c\in\OO$ has at least two solutions, then it has an invertible solution.
\end{cor}
\begin{proof} Assume that $X\subset \OO$ is the set of all solutions of a linear monomial equation $w(a_1,\ldots,a_m,x)=c$. Since $|X|>1$, there exist $a_i\in\OO_{\#}$ for some $1\leq i\leq m$ (see Corollary~\ref{cor_minom_eq}). Therefore, case 2 of the proof of Corollary~\ref{cor_minom_eq} holds. Hence, we repeat the reasoning from case 2 of the proof of Corollary~\ref{cor_minom_eq} and use notations from it. One of the following two cases holds:

\begin{enumerate}
\item[$\bullet$] Assume that equality~(\ref{eq_canon1}) holds.  Then $x_0 \in X_1$ and $Y_1$ is also nonempty. Since $c\neq0$, we have that $a_{i_r} v(x_0)$ is also nonzero. Hence, since $a_{i_r}\in \OO_{\#}$, it is easy to see that Proposition~\ref{prop_AXisB} implies that there exists an invertible $y_0\in Y_1$. Finally, $f_1(y_{\corr{0}})\in X_1\subset X$ is invertible. 

\item[$\bullet$] Assume that equality~(\ref{eq_canon2}) holds.  Then $x_0 \in X_2$ and $Y_2$ is also nonempty. Since $c\neq0$, we have that $v(x_0) a_{i_r}$ is also nonzero. Applying the involution $\ov{\stackrel{\;\;}{\;\;}}:\OO\to \OO$ to the equation $y a_{i_r}  =  v(x_0) a_{i_r}$, as above, we obtain that  there exists an invertible $y_0\in Y_2$, since $a_{i_r}\in \OO_{\#}$.  Finally, $f_2(y_{\corr{0}})\in X_2\subset X$ is invertible. 
\end{enumerate}
\end{proof}

\begin{remark}\label{rem_invertible}
In the statement of Corollary~\ref{cor_invertible}, the assumption that $c$ is nonzero is essential and cannot be omitted.

Indeed, consider the linear equation $ax = 0$ with $a \in \OO_{\#}$ nonzero. By Corollary~\ref{cor_AXisB}, its solution set $X$ is contained in $\Omega_4$ and is therefore infinite. On the other hand, if $X$ contained an invertible element $x_0$, then we would have $0 = ax_0 = (ax_0)x_0^{-1} = a$,
by equality~(\ref{eq_inv}), which contradicts the assumption $a \neq 0$.
\end{remark}

\medskip

Recall that an octonion $a \in \OO$ is said to be \emph{nilpotent} if $a^n = 0$ for some $n > 0$. 
Proposition~\ref{prop_O_canon}, Lemma~3.1 of~\cite{Lopatin_Rybalov_2025}, and equality~(\ref{eq2}) together imply that the following three conditions for $a \in \OO$ are equivalent:
\begin{equation}
a \text{ is nilpotent }
\;\;\Longleftrightarrow\;\;
\tr(a) = n(a) = 0
\;\;\Longleftrightarrow\;\;
a^2 = 0.
\end{equation}

\begin{remark}\label{rem_nilpotent}
In the statement of Corollary~\ref{cor_invertible}, a nilpotent solution cannot be considered in place of an invertible one.

Indeed, consider the linear equation $ax = a$ with $a \in \OO_{\#}$ nonzero. Since 
$\overline{a}a = \overline{a}\,\overline{\overline{a}} = n(\overline{a}) = n(a) = 0$,
Corollary~\ref{cor_AXisB} implies that the set of its solutions $X$ is contained in $\Omega_4$ and is therefore infinite. On the other hand, if $X$ contained a nilpotent element $x_0$, then we would have $0 = a \cdot x_0^2 = (ax_0)x_0 = ax_0 = a$, by equality~(\ref{eq4}), which contradicts the assumption that $a \neq 0$.
\end{remark}


\section{Degenerations of linear monomial equations}\label{section_degenerate}

Given $a_1,\ldots,a_m,c\in\OO$ and $a'_1,\ldots,a'_m,c'\in\OO$, the linear monomial equation $w(a'_1,\ldots,a'_m,x)=c'$ is called a {\it degeneration} of the linear monomial equation $w(a_1,\ldots,a_m,x)=c$, if the Zariski closure of the $\G$-orbit of $(a_1,\ldots,a_m,c)\in\OO^{m+1}$ contains $(a'_1,\ldots,a'_m,c')$.

\begin{cor}\label{cor_degenerate} Let $w(a'_1,\ldots,a'_m,x)=c'$ be a degeneration of  $w(a_1,\ldots,a_m,x)=c$, where $a_1,\ldots,a_m,a'_1,\ldots,a'_m\in\OO$ are nonzero. Let $X'$ ($X$, respectively) be the set of all  solutions of the first (the second, respectively) equation. Then 
\begin{enumerate}
\item[(a)] the first equation has a unique solution if and only if the second equation has a unique solution;

\item[(b)] in case $w(a_1,x) = a_1 x$ we have that $\dim X' \corr{\geq} \dim X$;

\item[(c)] in case $w(a_1,a_2,x) = (a_1x)a_2$ we have that \corr{if} $\dim X = 4$, \corr{then} $ \dim X' =4$;

\item[(d)] in case $w(a_1,a_2,x) = a_1(a_2x)$ we have that \corr{if} $X = \OO$, \corr{then} $X' =\OO$.
\end{enumerate}
\end{cor}
\begin{proof} Since $n(a_i)=n(a'_i)$ for all $1\leq i\leq m$, then  
part (a) follows from Corollary~\ref{cor_minom_eq}, 
part (b) follows from Corollary~\ref{cor_AXisB},
part (c) follows from Corollary~\ref{cor_AXBisC}, 
part (d) follows from Corollary~\ref{cor_ABXisC}.
\end{proof}

\section{Conclusion}\label{section_conclusion}

In this manuscript, over an algebraically closed field, we described the affine varieties of solutions to the linear equations \corr{$(ax)b=c$} and $a(bx)=c$ over the algebra of split-octonions $\OO$, which is non-division. Moreover, we showed that if a linear monomial equation over the split-octonions with a nonzero constant term has at least two solutions, then it necessarily possesses an invertible solution.

Possible directions for further research include solving the following equations.

\begin{enumerate}
\item[$\bullet$] Special equations over $\OO$, such as the generalized Sylvester equation $ax - yb = c$, the equation $(\ov{x}a)x=b$, the quadratic equation $x^2 + ax + b = c$, and the system of equations $ax=c$ and $xb=d$, where $x,y\in\OO$ are variables and $a,b,c,d\in\OO$ are given constants.

\item[$\bullet$] The general linear equation over $\OO$;

\item[$\bullet$] Polynomial equations over split-octonions. The first step in this direction has already been made in~\cite{Lopatin_Rybalov_2025}, where every polynomial equation
\[
\alpha_n x^n + \cdots + \alpha_1 x + a = 0,
\]
with $\alpha_1,\ldots,\alpha_n\in\FF$ and $a\in\OO$, was solved over split-octonions when $\FF$ is algebraically closed. This result was later generalized to the case of an arbitrary field in~\cite{Lopatin_Eq_over_O_III}.
\end{enumerate}


\bigskip

\bibliographystyle{abbrvurl}
\bibliography{main}

@article{zubkov2018,
    AUTHOR = {Zubkov, Alexandr N. and Shestakov, Ivan},
     TITLE = {Invariants of {${\rm G}_2$} and {${\rm Spin}(7)$}  in positive characteristic},
  JOURNAL = {Transformation Groups},
    VOLUME = {23},
      YEAR = {2018},
    number = {2},
     PAGES = {555--588},
       doi = {10.1007/s00031-017-9435-8}
}

@article{schwarz1988,
    AUTHOR = {G.W. Schwarz},
     TITLE = {Invariant theory of {${\rm G}_2$} and {${\rm Spin}_7$}},
  JOURNAL = {Comment. Math. Helvetici},
    VOLUME = {63},
      YEAR = {1988},
     PAGES = {624--663},
      doi = {10.1007/BF02566782}
}

@article{LZ_1,
    AUTHOR = {Lopatin, Artem and Zubkov, Alexandr N.},
     TITLE = {Classification of {${\mathrm G}_2$}-orbits for pairs of octonions},
  JOURNAL = {Journal of Pure and Applied Algebra},
    VOLUME = {229},
      YEAR = {2025},
     PAGES = {107875},
       doi = {https://doi.org/10.1016/j.jpaa.2025.107875}
}

@article{LZ_2,
    AUTHOR = {Lopatin, Artem and Zubkov, Alexandr N.},
     TITLE = {Separating {${\mathrm G}_2$}-invariants of several octonions},
  JOURNAL = {Algebra Number Theory},
    VOLUME = {18},
      YEAR = {2024},
    number = {12},
     PAGES = {2157--2177},
      doi = {10.2140/ant.2024.18.2157}
}

@article{Chapman_2020_JAA,
      AUTHOR = {Chapman, Adam},
     TITLE = {Polynomial equations over octonion algebras},
  JOURNAL = {Journal of Algebra and Its Applications},
    VOLUME = {19},
      YEAR = {2020},
    number = {6},
     PAGES = {2050102},
       doi = {10.1142/S0219498820501029}
}

@article{Chapman_Vishkautsan_2022,
    AUTHOR = {Chapman, Adam and Vishkautsan, Solomon},
     TITLE = {Roots and dynamics of octonion polynomials},
  JOURNAL = {Communications in Mathematics},
    VOLUME = {30},
      YEAR = {2022},
    number = {2},
     PAGES = {25--36},
       doi = {10.46298/cm.9042}
}

@article{Chapman_Levin_2023,
    AUTHOR = {Chapman, Adam and Levin, Ilan},
     TITLE = {Alternating roots of polynomials over {C}ayley-{D}ickson algebras},
  JOURNAL = {Communications in Mathematics},
    VOLUME = {32},
      YEAR = {2024},
    number = {2},
     PAGES = {63--70},
       doi = {10.46298/cm.11514}
}

@article{Chapman_Vishkautsan_2025,
    AUTHOR = {Chapman, Adam and Vishkautsan, Solomon},
     TITLE = {Roots and right factors of polynomials and left eigenvalues of matrices over {C}ayley-{D}ickson algebras},
  JOURNAL = {Communications in Mathematics},
    VOLUME = {33},
      YEAR = {2025},
    number = {3},
     PAGES = {~paper no.~1},
       doi = {10.46298/cm.12613}
}

@article{Flaut_Shpakivskyi_2015,
    AUTHOR = {Flaut, Cristina and Shpakivskyi, Vitalii},
     TITLE = {An efficient method for solving equations in generalized quaternion and octonion algebras},
  JOURNAL = {Advances in Applied Clifford Algebras},
    VOLUME = {25},
      YEAR = {2015},
    number = {},
     PAGES = {337--350},
       doi = {10.1007/s00006-014-0493-x}
}

@article{Wang_Zhang_2014,
    AUTHOR = {Wang, Qing-Wen and Zhang, Xiang and Zhang, Yang},
     TITLE = {Algorithms for finding the roots of some quadratic octonion equations},
  JOURNAL = {Communications in Algebra},
    VOLUME = {42},
      YEAR = {2014},
    number = {8},
     PAGES = {3267--3282},
       doi = {10.1080/00927872.2013.780062}
}

@book{Springer_Veldkamp_book_2000,
    AUTHOR = {Springer, T.A. and Veldkamp, F.D.},
     TITLE = {Octonions, Jordan algebras and exceptional groups},
 PUBLISHER = {Springer-Verlag, Berlin},
      YEAR = {2000},
       doi = {10.1007/978-3-662-12622-6}
}

@article{Illmer_Netzer_2024,
    AUTHOR = {Illmer, M. and Netzer, T.},
     TITLE = {A note on polynomial equations over algebras},
  JOURNAL = {Proceedings of the American Mathematical Society},
    VOLUME = {152},
      YEAR = {2024},
     PAGES = {1831--1839},
       doi = {10.1090/proc/16630}
}

@article{Rodriguez-Ordonez_2007,
     AUTHOR = {Rodríguez-Ordó\~nez, H.},
     TITLE = {A note on the fundamental theorem of algebra
for the octonions},
  JOURNAL = {Expo. Math.},
    VOLUME = {25},
      YEAR = {2007},
     PAGES = {355--361},
       doi = {10.1016/j.exmath.2007.02.005}
}

@article{Chanyal_2017_CommTP,
    AUTHOR = {Chanyal, B.C.},
     TITLE = {Split octonion reformulation for electromagnetic chiral media of massive dyons},
  JOURNAL = {Communications in Theoretical Physics (Beijing)},
    VOLUME = {68},
      YEAR = {2017},
    number = {6},
     PAGES = {701--710},
       doi = {10.1088/0253-6102/68/6/701}
}

@article{Chanyal_Bisht_Negi_2013,
    AUTHOR = {Chanyal, B.C. and Bisht, P.S. and Negi, O.P.S.},
     TITLE = {Octonion and conservation laws for dyons},
  JOURNAL = {International Journal of Modern Physics A},
    VOLUME = {28},
      YEAR = {2013},
    number = {26},
     PAGES = {~paper no.~1350125, 17 pp.},
       doi = {10.1142/S0217751X1350125X}
}

@article{Chanyal_Bisht_Negi_2011,
    AUTHOR = {Chanyal, B.C. and Bisht, P.S. and Negi, O.P.S.},
     TITLE = {Generalized split-octonion electrodynamics},
  JOURNAL = { International Journal of Theoretical Physics},
    VOLUME = {50},
      YEAR = {2011},
    number = {6},
     PAGES = {1919--1926},
       doi = {10.1007/s10773-011-0706-1}
}

@article{Chanyal_2015_RepMP,
    AUTHOR = {Chanyal, B.C.},
     TITLE = {Classical geometrodynamics with {Z}orn vector-matrix algebra for gravito-dyons},
  JOURNAL = {Reports on Mathematical Physics},
    VOLUME = {76},
      YEAR = {2015},
    number = {1},
     PAGES = {1--20},
       doi = {10.1016/S0034-4877(15)00025-7}
}

@article{Krasnov_2022,
    AUTHOR = {Krasnov, Kirill},
     TITLE = {Spin(11,3), particles, and octonions},
  JOURNAL = {Journal of Mathematical Physics},
    VOLUME = {63},
      YEAR = {2022},
    number = {3},
     PAGES = {~paper no.~031701, 20 pp.},
       doi = {https://doi.org/10.1063/5.0070058}
}

@article{Bisht_Dangwal_Negi_2008,
    AUTHOR = {Bisht, P.S. and Dangwal, Shalini and Negi, O.P.S.},
     TITLE = {Unified split octonion formulation of dyons},
   JOURNAL = {International Journal of Theoretical Physics},
    VOLUME = {47},
      YEAR = {2008},
    number = {9},
     PAGES = {2297--2313},
       doi = {10.1007/s10773-008-9662-9}
}

@article{Castro_2007,
    AUTHOR = {Castro, Carlos},
     TITLE = {On the noncommutative and nonassociative geometry of octonionic space time, modified dispersion relations and grand unification},
  JOURNAL = {Journal of Mathematical Physics},
    VOLUME = {48},
      YEAR = {2007},
    number = {7},
     PAGES = {~paper no.~073517, 15 pp.},
       doi = {10.1063/1.2752013}
}

@article{Gogberashvili_2006,
    AUTHOR = {Gogberashvili, M},
     TITLE = {Octonionic electrodynamics},
  JOURNAL = {Journal of Physics A: Mathematical and General},
    VOLUME = {39},
      YEAR = {2006},
    number = {22},
     PAGES = {7099--7104},
       doi = {10.1088/0305-4470/39/22/020}
}

@article{Gogberashvili_2006_Dirac,
    AUTHOR = {Gogberashvili, Merab},
     TITLE = {Octonionic version of {D}irac equations},
  JOURNAL = {International Journal of Modern Physics A},
    VOLUME = {21},
      YEAR = {2006},
    number = {17},
     PAGES = {3513--3523},
       doi = {10.1142/S0217751X06028436}
}

@article{Koplinger_2006,
    AUTHOR = {Köplinger, Jens},
     TITLE = {Dirac equation on hyperbolic octonions},
  JOURNAL = {Applied Mathematics and Computation},
    VOLUME = {182},
      YEAR = {2006},
    number = {1},
     PAGES = {443--446},
       doi = {10.1016/j.amc.2006.04.005},
     note = {Corrigendum in \cite{Koplinger_2007} footnote 2}
}

@article{Koplinger_2007,
    AUTHOR = {Köplinger, Jens},
     TITLE = {Gravity and electromagnetism on conic sedenions},
  JOURNAL = {Applied Mathematics and Computation},
    VOLUME = {188},
      YEAR = {2007},
    number = {1},
     PAGES = {948--953},
       doi = {10.1016/j.amc.2006.10.050}
}

@article{Gogberashvili_Sakhelashvili_2015,
    AUTHOR = {Gogberashvili, Merab and Sakhelashvili, Otar},
     TITLE = {Geometrical applications of split octonions},
  JOURNAL = {Advances in Mathematical Physics},
    VOLUME = {},
      YEAR = {2015},
    number = {},
     PAGES = {art. ID 196708, 14 pp.},
       doi = {10.1155/2015/196708}
}

@article{Eq_BK_79,
    AUTHOR = {Baksalary, J.K. and Kala, R.},
     TITLE = {The matrix equation ${AX-YB=C}$},
   JOURNAL = {Linear Algebra Appl.},
    VOLUME = {25},
      YEAR = {1979},
     PAGES = {41--43},
       doi = {10.1016/0024-3795(79)90004-1}
}

@article{Eq_BK_80,
    AUTHOR = {Baksalary, J.K. and Kala, R.},
     TITLE = {The matrix equation ${AXB+CYD=E}$},
   JOURNAL = { Linear Algebra Appl.},
    VOLUME = {30},
      YEAR = {1980},
     PAGES = {141--147},
       doi = {10.1016/0024-3795(80)90189-5}
}

@article{Eq_T_00,
    AUTHOR = {Tian, Yongge},
     TITLE = {The solvability of two linear matrix equations},
   JOURNAL = {Linear and Multilinear Algebra },
    VOLUME = {48},
      YEAR = {2000},
    number = {2},
     PAGES = {123--147},
       doi = {10.1080/03081080008818664}
}

@article{Eq_L_06,
    AUTHOR = {Liu, Yong Hui},
     TITLE = {Ranks of solutions of the linear matrix equation ${AX+YB=C}$},
   JOURNAL = {Comput. Math. Appl.},
    VOLUME = {52},
      YEAR = {2006},
    number = {6--7},
     PAGES = {861--872},
       doi = {10.1016/j.camwa.2006.05.011}
}

@article{Eq_FLLT_19,
    AUTHOR = {Ferreyra, D.E. and Lattanzi, M. and Levis, F.E. and Thome, N.},
     TITLE = {Parameterized solutions ${X}$ of the system ${AXA=AEA}$ and ${A^kEAX=XAEA^k}$ for a matrix ${A}$ having index $k$},
   JOURNAL = {Electron. J. Linear Algebra},
    VOLUME = {35},
      YEAR = {2019},
     PAGES = {503--510},
       doi = {10.13001/1081-3810.4051}
}

@article{Eq_B_99,
    AUTHOR = {Braden, H.W.},
     TITLE = {The equations ${A^TX\pm X^TA=B}$},
   JOURNAL = {SIAM J. Matrix Anal. Appl. },
    VOLUME = {20},
      YEAR = {1999},
    number = {2},
     PAGES = {295--302},
       doi = {10.1137/S0895479897323270}
}

@article{Lopatin_Rybalov_2025,
    AUTHOR = {Lopatin, Artem and Rybalov, Alexander N.},
     TITLE = {On polynomial equations over split octonions},
   JOURNAL = {Communications in Mathematics},
    VOLUME = {33},
      YEAR = {2025},
    number = {3},
     PAGES = {~paper no.~8},
       doi = {10.46298/cm.14879}
}

@article{Lopatin_Eq_over_O_III,
    AUTHOR = {Lopatin, Artem},
     TITLE = {On polynomial equations over split-octonions: the arbitrary field case},
   JOURNAL = {Communications in Mathematics},
    VOLUME = {34},
      YEAR = {2026},
    number = {2},
     PAGES = {~paper no.~12},
       doi = {10.46298/cm.17309}
}

@article{wang1,
    AUTHOR = {Wang, Qing-Wen and Xie, Lv-Ming and Gao, Zi-Han},
     TITLE = {A survey on solving the matrix equation ${AXB = C}$ with applications},
   JOURNAL = {Mathematics},
    VOLUME = {13},
      YEAR = {2025},
    number = {3},
     PAGES = {~paper no.~450},
       doi = {10.3390/math13030450}
}

@article{wang2,
    AUTHOR = {Wang, Qing-Wen and Gao, Jiale},
     TITLE = {Comprehensive review on the generalized {S}ylvester equation ${AX-YB = C}$},
   JOURNAL = {Symmetry},
    VOLUME = {17},
      YEAR = {2025},
    number = {10},
     PAGES = {~paper no.~1686},
       doi = {10.3390/sym17101686}
}

@article{wang3,
    AUTHOR = {Wang, Qing-Wen and Gao, Zi-Han and Gao, Jia-Le},
     TITLE = {A comprehensive review on solving the system of equations ${AX = C}$ and ${XB = D}$},
   JOURNAL = {Symmetry},
    VOLUME = {17},
      YEAR = {2025},
    number = {4},
     PAGES = {~paper no.~625},
       doi = {10.3390/sym17040625}
}

@article{wang4,
    AUTHOR = {Wang, Qing-Wen and Gao, Zi-Han and  Li, Yu-Fei},
     TITLE = {Overview of methods for solving the system of equations ${A_1XB_1 = C_1}$ and ${A_2XB_2 = C_2}$.},
   JOURNAL = {Symmetry},
    VOLUME = {17},
      YEAR = {2025},
    number = {8},
     PAGES = {~paper no.~1307},
       doi = {10.3390/sym17081307}
}

\end{document}